\documentclass[11pt]{amsart}
\usepackage{amsfonts,amssymb,amsmath,amsgen,amsthm}
\usepackage{hyperref}

\usepackage{geometry}                
\usepackage{color}

\def\d{{\partial}}
\def\eps{\varepsilon}

\def\C{{\mathbb C}}
\def\R{{\mathbb R}}
\def\N{{\mathbb N}}
\def\Sch{{\mathcal S}}

\def\Re{\text{\rm Re}}

\def\Tend#1#2{\mathop{\longrightarrow}\limits_{#1\rightarrow#2}}
\def\Tendweak#1#2{\mathop{\rightharpoonup}\limits_{#1\rightarrow#2}}

\def\<{{\langle}}
\def\>{{\rangle}}
\def\({\left(}
\def\){\right)}

\theoremstyle{plain}
\newtheorem{theorem}{Theorem}[section]
\newtheorem{definition}[theorem]{Definition}

\newtheorem{lemma}[theorem]{Lemma}

\newtheorem{proposition}[theorem]{Proposition}

\theoremstyle{remark}
\newtheorem{remark}[theorem]{Remark}

\newtheorem{example}[theorem]{Example}

\def\ech{\color{black}\,}

\numberwithin{equation}{section}

\begin{document}

\title[Hartree system]{On the Cauchy problem for the Hartree
  approximation in quantum dynamics}

\author[R. Carles]{R\'emi Carles}
\address[R. Carles]{Univ Rennes, CNRS\\ IRMAR - UMR 6625\\ F-35000
  Rennes, France}
\email{Remi.Carles@math.cnrs.fr}

\author[C. Fermanian]{Clotilde Fermanian Kammerer}
\address[C. Fermanian Kammerer]{
Univ Paris Est Creteil, CNRS, LAMA, F-94010 Creteil, France\\ 
Univ Gustave Eiffel, LAMA, F-77447 Marne-la-Vallée, France}
\email{clotilde.fermanian@u-pec.fr}

\author[C. Lasser]{Caroline Lasser}
\address[C. Lasser]{Technische Universit\"at M\"unchen, Zentrum Mathematik, Deutschland}
\email{classer@ma.tum.de}

\begin{abstract}
We prove existence and uniqueness results for  the time-dependent Hartree approximation 
arising in quantum dynamics. The Hartree equations of motion form a coupled system
of nonlinear  Schr\"odinger equations  for the evolution of product state approximations. 
They are a prominent example for dimension reduction in the context of the the time-dependent 
Dirac--Frenkel variational principle. 
  We handle  the case of Coulomb potentials thanks to Strichartz
estimates. Our main result addresses a general setting where the 
nonlinear coupling cannot be considered as a perturbation.
 The proof uses a recursive construction that  is inspired by the 
standard approach for the Cauchy problem associated to symmetric quasilinear hyperbolic equations.
\end{abstract}

\thanks{RC was supported by Rennes M\'etropole through its AIS
  program, and by Centre Henri Lebesgue, program ANR-11-LABX-0020-0. 
   CL was supported by the Centre for Advanced Study in Oslo, Norway, research project
\emph{Attosecond Quantum Dynamics Beyond the Born-Oppenheimer Approximation\ech}.
}

\maketitle


\section{Introduction}
\label{sec:intro}

We  consider the time-dependent Schr\"odinger equation
\begin{equation}
  \label{eq:exact}
  i\d_t \psi = H\psi, 
\end{equation}
where the total Hamiltonian is given by
\begin{equation*}
   H=H_x+H_y+w(x,y), \quad H_x  =- \frac 12 \Delta_x +V_1(x) ,\quad H_y
  = -\frac 12 \Delta_y +V_2(y)
\end{equation*}
with $x\in \R^{d_1}$ and $y\in \R^{d_2}$, $d_1,d_2\ge 1$. 
The potentials $V_1,V_2$ and $w$ are always real-valued, we will make
extra regularity and decay assumptions later on\ech. 
It is common wisdom that for dealing with quantum systems ``a solution of the 
wave equation in many-dimensional space is far too complicated to be practicable'' 
(Dirac 1930) and one aims at approximative methods that effectively reduce the space dimension. 
Here, we focus on initial data that decouple the space variables,
\[\psi(0,x,y)=\phi^x_0(x)\phi^y_0(y).\]
Such a form of initial data indeed  suggests a dimension reduction approach. Of course, if there is no coupling ($w(x,y)=0$), the full solution is itself a product state $\psi(t,x,y)= \phi_x(t,x)\phi_y(t,y)$, with 
\begin{equation*}
\begin{cases}
i\partial_t \phi^x= H_x \phi^x,\quad\phi^x(0,x)=\phi^x_0(x),\\
i\partial_t \phi^y= H_y \phi^y ,\quad\phi^y(0,y)=\phi^y_0(y).
\end{cases}
\end{equation*}
When the coupling is present, one seeks for approximate solutions $u(t)\approx\psi(t)$ of product form $u(t,x,y)= \phi^x(t,x)\phi^y(t,y)$ in order to reduce the initial system~\eqref{eq:exact}  in $\R^{d_1+d_2}$ to two systems on spaces of smaller dimensions, $\R^{d_1}$ and $\R^{d_2}$. In situations, where the overall configuration space has a natural decomposition of its dimension $d_1 + \cdots + d_N$, a corresponding product ansatz of $N$ factors is sought. Here we only investigate the case $N=2$, mentioning that repeated application of the binary construction yields the more general case. \ech
Applying the time-dependent Dirac--Frenkel variational principle to the manifold
\[
\mathcal M = \left\{u = \varphi^x\otimes\varphi^y\mid 
\varphi^x\in L^2(\R^{d_1}),\ \varphi^y\in L^2(\R^{d_2})\right\}
\]
yields the so-called {\it time-dependent Hartree approximation},
\[
\psi(t,x,y)\approx \phi^x(t,x)\phi^y(t,y)\in\mathcal M,
\]
 where the pair
$(\phi^x,\phi^y)$ solves the nonlinearly coupled system
\begin{equation}\label{eq:system}
\begin{cases}
i\partial_t \phi^x= H_x \phi^x + \langle w\rangle_y \phi^x,\quad\phi^x(0,x)=\phi^x_0(x),\\
i\partial_t \phi^y= H_y \phi^y + \langle w\rangle_x \phi^y,\quad\phi^y(0,y)=\phi^y_0(y).
\end{cases}
\end{equation}
The time-dependent potentials  result from the averaging process
\begin{align}
\nonumber
&\langle w\rangle_y (t,x) :=\int_{\R^{d_2}} w(x,y) |\phi^y(t,y)|^2 dy ,\\
\nonumber
 &\langle w\rangle_x(t,y) :=\int_{\R^{d_1}} w(x,y) |\phi^x(t,x)|^2 dx,
\end{align}
under the assumption, made throughout this paper, that
\begin{equation}\label{eq:normL2}
  \|\phi_0^x\|_{L^2(\R^{d_1})} =  \|\phi_0^y\|_{L^2(\R^{d_2})} =1.
\end{equation}
For  any ``reasonable'' solution (at least with the regularity
considered in this paper), the $L^2$-norms of $\phi^x(t,\cdot)$ and
$\phi^y(t,\cdot)$, respectively, are independent of time, hence
\begin{equation*}
   \|\phi^x(t)\|_{L^2(\R^{d_1})} =  \|\phi^y(t)\|_{L^2(\R^{d_2})} =1,
 \end{equation*}
 for all $t$ in the time interval where the solution to
 \eqref{eq:system} is well-defined; see \S\ref{sec:conservations} for a proof.\ech
 \medskip 
 
 Even though the time-dependent Hartree approximation is one of the most fundamental 
 approximations in quantum dynamics, mathematical existence and uniqueness proofs are rather scarce. \ech
 Existence and uniqueness have been studied in the case where the interaction potential is of convolution type, i.e. for  
 $w(x,y)=W(x-y)$ and with one of the subsystems moving by classical mechanics (see~\cite{CancesLeBris99,Baudouin2005,CacciafestaSuzzoniNoja2020} for example).
 A related investigation has targeted the time-dependent self-consistent field system~\cite{JinSparberZhou2017}
with coupling potentials of Schwartz class. However,   
 our aim here is to discuss the existence and uniqueness of solutions
 for system~\eqref{eq:system} when the potentials $\langle w\rangle_x$
 or $\langle w\rangle_y$  need not be bounded, and cannot be considered as a perturbation of $V_2$ or $V_1$, respectively.
  This framework  requires a different approach. In particular, our result 
  provides the Cauchy theory for the systems  discussed in the articles~\cite{BCFLL21,BCFLL2} where the 
  accuracy of the Hartree approximation is studied in the broader context of composite quantum dynamics and scale separation.  
  
The steps for our existence and uniqueness proof are strongly inspired by the method which is
classical in the study of quasilinear hyperbolic systems, see e.g. \cite{AlGe07}.
With $n\in \N$, we associate the iterative scheme of recursive equations 
 \begin{equation}
\label{eq:scheme}
\left\{
\begin{aligned}
i\partial_t \phi^x_{n+1}= H_x \phi^x_{n+1} + \langle w_n\rangle_y \phi^x_{n+1},\quad \phi^x_{n+1}(0,x)= \phi^x_0(x),\\
i\partial_t \phi^y_{n+1}= H_y \phi^y_{n+1}  + \langle w_n\rangle_x
\phi^y_{n+1} ,\quad \phi^y_{n+1}(0,y)= \phi^y_0(y),
\end{aligned}
\right.
\end{equation}
with 
 \begin{equation}
\label{eq:av-scheme}
\langle w_n\rangle_y(t,x)=\int_{\R^n}w(x,y) |\phi^y_n(t,y)|^2 dy,\quad
   \langle w_n\rangle_x(t,y)=\int_{\R^d} w(x,y) |\phi^x_n(t,x)|^2 dx.
\end{equation}
The main steps of the proof of our existence and uniqueness result are then:
\begin{enumerate}
\item  The iterative scheme is well-defined  and enjoys  bounds in
  ``large'' norm, which control second order derivatives and polynomial growth of order two for some finite time horizon.\ech
\item The solution of the scheme converges in ``small''  norm, which is the $L^2$ norm.  
\item It is possible to pass to the limit $n\to+\infty$ in the equation, which leads to the construction of a solution that one then proves to be unique and global in time (provided  the initial data is regular enough). 
\end{enumerate}

\subsection{Outline}
In the next Section~\ref{sec:var}, we recall elementary properties of the time-dependent variational principle and formally derive the Hartree equations \eqref{eq:system}. Then we discuss coupling potentials $w(x,y)$ of Coulombic and of polynomial type in Section~\ref{sec:result}, where we also 
present our main result Theorem~\ref{theo:main}, which establishes existence and uniqueness of the 
solutions to the Hartree system for coupling with polynomial growth. The different steps of the proof of  Theorem~\ref{theo:main} are the subject of Section~\ref{sec:scheme} (analysis of the iterative scheme), Section~\ref{sec:convergence} (convergence in small norms) and Section~\ref{sec:solution} (passing to the limit). A sufficient condition for the growth of the coupling potential is verified in Section~\ref{sec:sufficient}.
The Appendices~\ref{sec:tangent} and \ref{sec:coulomb} summarize some technical arguments. \ech

\subsection{Notations}
We write $L^\infty_T$ for $L^\infty([0,T])$. The notations $L^2_x$, $L^2_y$, $L^2_{x,y}$ stand
for $L^2(\R^{d_1}_x)$, $L^2(\R^{d_2}_y)$, $L^2(\R^{d_1+d_2}_{x,y})$, respectively. We denote by $\<\cdot,\cdot\>_{L^2_x}$, $\<\cdot,\cdot\>_{L^2_y}$, $\<\cdot,\cdot\>_{L^2_{x,y}}$ the corresponding inner products.
 For $f,g\ge0$, we write $f\lesssim g$ whenever there exists a
``universal'' constant (in the sense that it does not depend on time, space,
or $n$, typically) such that $f\le Cg$.


\section{Variational principle}\label{sec:var}
The time-dependent Hartree approximation results from the Dirac--Frenkel variational principle applied to the manifold
\begin{equation}\label{def:M}
\mathcal M = \left\{u = \varphi^x\otimes\varphi^y\mid 
\varphi^x\in L^2_x,\ \varphi^y\in L^2_y\right\},
\end{equation}
 see also \cite[\S II.3.1]{LubichBlue} for the analogous discussion with Hartree products of $N$ orbitals 
in $L^2(\R^3)$. The reader can also refer to \cite{LasserSu2022}. The principle determines an approximate solution 
$u(t)\in\mathcal M$ for the time-dependent Schr\"odinger equation
\[
i\partial_t\psi = H\psi
\]
with initial data $\psi(0)\in\mathcal M$ by requiring that for all times $t$
\begin{equation}\label{eq:var_cond}
\left\{
\begin{array}{l}
\partial_t u(t)\in\mathcal T_{u(t)}\mathcal M,\\*[1ex]
\langle v, i\partial_t u(t) - Hu(t)\rangle = 0\quad\text{for all}\quad v\in\mathcal T_{u(t)}\mathcal M,
\end{array}
\right.
\end{equation}
 where $\mathcal T_{u(t)}\mathcal M$ denotes the tangent space of $\mathcal M$ at $u(t)$. \ech
For deriving the Hartree equations, we first have to understand the manifold $\mathcal M$ and its tangent space. Note that the representation of a Hartree function is non-unique, since $\varphi^x \otimes\varphi^y = 
(a\varphi^x)\otimes(a^{-1}\varphi^y)$ for any $a\in\C\setminus\{0\}$. However, 
we can have unique representations in the tangent space once appropriate gauge conditions are set. 

\begin{lemma}[Tangent space]\label{lem:tangent} For any $u = \varphi^x\otimes\varphi^y\in\mathcal M$, $u\neq 0$,
\[
\mathcal T_u\mathcal M = \left\{v^x\otimes\varphi^y + \varphi^x\otimes v^y\mid v^x\in L^2_x,\ v^y\in L^2_y\right\}.
\]
Any $v\in\mathcal T_u\mathcal M$ has a unique representation of the form $v=v^x\otimes\varphi^y + \varphi^x\otimes v^y$, if we impose the gauge condition
$\langle \varphi^x,v^x \rangle = 0$. The tangent spaces are complex linear subspaces of $L^2_{x,y}$ such that $u\in\mathcal T_u\mathcal M$ for all $u\in\mathcal M$.
\end{lemma}

The lemma is proved in Appendix~\ref{sec:tangent}. The following  formal arguments show that, in case that the variational solution $u(t)$ is well-defined and sufficiently regular,  the $L^2$ norm and the energy expectation value  are conserved  automatically. Indeed,\ech
we differentiate with respect to time $t$ and use the variational condition~\eqref{eq:var_cond} for $v=u(t)$, 
\[
\frac{d}{dt}\|u(t)\|^2_{L^2_{x,y}} = 2\,\mathrm{Re}\<u(t),\partial_t u(t)\>_{L^2_{x,y}} = 
2\,\mathrm{Re}\< u(t),\tfrac{1}{i} Hu(t)\>_{L^2_{x,y}} = 0,
\]
due to self-adjointness of the Hamiltonian. Similarly, using self-adjointness and the variational condition~\eqref{eq:var_cond} for 
$v=\partial_tu(t)$, 
\[
\frac{d}{dt}\langle u(t),Hu(t)\rangle_{L^2_{x,y}} = 2\,\mathrm{Re}\langle \partial_t u(t),H u(t)\rangle_{L^2_{x,y}} = 
2\,\mathrm{Re}\langle \partial_t u(t),i\partial_t u(t)\rangle_{L^2_{x,y}} = 0.
\] 
Let us now  formally derive the Hartree system~\eqref{eq:system}. We write 
\[
u(t)=\varphi^x(t)\otimes \varphi^y(t),
\]
 with $\| \varphi^x(t)\|_{L^2_x}=\| \varphi^y(t)\|_{L^2_y}=1$.  
 We have 
\begin{align*}
i\partial_ t u & =( i\partial_t \varphi^x(t))\otimes \varphi^y(t) + \varphi^x(t)\otimes (i\partial_t  \varphi^y(t) ),\\
Hu & = H_x \varphi^x(t))\otimes \varphi^y(t) + \varphi^x(t)\otimes (H_y \varphi^y(t) ) + w(x,y) \varphi^x(t) \otimes \varphi^y(t).
\end{align*}
Choosing elements  $v=v^x\otimes\varphi^y + \varphi^x\otimes v^y\in\mathcal T_{u(t)}\mathcal M$ and evaluating~\eqref{eq:var_cond}, we obtain the following necessary and sufficient conditions:\ech
\begin{itemize}
\item[(i)] If $v^y$=0, we obtain that for all $v^x\in L^2_x$ such that $\langle v^x ,\varphi^x(t)\rangle_{L^2_x}=0$,
\[
\langle v^x, (i\partial_t -H_x) \varphi^x(t)\rangle_{L^2_{x}} = \int_{\R^{d_1+d_2}} w(x,y)v^x(x)\varphi^x(t,x) |\varphi^y(t,y)|^2 dxdy.
\]
\item[(ii)] If $v^x$=0, we obtain that for all $v^y\in L^2_y$ such that $\langle v^y ,\varphi^y(t)\rangle_{L^2_y}=0$,
\[
\langle v^y, (i\partial_t -H_y) \varphi^y(t)\rangle_{L^2_y} = \int_{\R^{d_1+d_2}} w(x,y)v^y(y)\varphi^y(t,y) |\varphi^x(t,x)|^2 dxdy.
\]
\end{itemize}
The choice of $
\varphi^x(t)$ and $\varphi^y(t)$ satisfying the Hartree system~\eqref{eq:system} guarantees (i) and (ii).

\section{Main result}\label{sec:result}
We present existence and uniqueness results for the solution of the time-dependent Hartree system \eqref{eq:system}. 
In \S\ref{subsec:coulomb}, we discuss how Strichartz estimates may be applied to Coulombic coupling. 
In \S\ref{subsec:poly}, we give detailed assumptions on polynomial growth conditions. Then, in \S\ref{sec:main-result} we state our 
main result Theorem~\ref{theo:main}. 
 
\subsection{Coupling potentials of Coulombic form}\label{subsec:coulomb}

The case of Coulomb singularities (in combination with classical nuclear dynamics) has already 
been addressed in~\cite{CancesLeBris99,Baudouin2005} by Schauder and Picard fixed point 
arguments, respectively.
We briefly revisit the main result from  \cite{CancesLeBris99}, and
show how it may be adapted thanks to 
Strichartz estimates.
We suppose $d_1=d_2=3$, and have in mind the case
\begin{equation}\label{eq:wCoulomb}
  w(x,y) =\frac{\eps}{|x-y|},\quad \eps\in \R.
\end{equation}
We consider more generally the case $w(x,y) = W(x-y)$, for a possibly
singular $W$.
We assume that the potentials $V_1$ and $V_2$ are perturbations of
smooth and at most quadratic potentials:
\begin{equation*}
  V_j = {\mathbf V}_j+ v_j,
\end{equation*}
where 
\begin{equation*}
{\mathbf V}_j\in \mathcal Q=\left\{V\in
                 C^\infty(\R^3;\R),\ \d^\alpha V\in L^\infty(\R^3),\
                 \forall |\alpha|\ge 2\right\},
\end{equation*}
and
\begin{equation*}
  v_j\in L^p(\R^3)+L^\infty(\R^3), \quad\text{for some }p>3/2.
\end{equation*}
Typically, we may consider Coulomb potentials, as 
\begin{equation*}
  \frac{1}{|x|}=\frac{1}{|x|}{\mathbf 1}_{|x|<1}+
 \frac{1}{|x|}{\mathbf 1}_{|x|\ge 1}.
\end{equation*}
The first term on the right hand side belongs to $L^p(\R^3)$ for any
$1\le p<3$, and the second term is obviously bounded. We then make the
same assumption on $W$.
Denote 
\begin{equation*}
  \mathbf H_x=-\frac{1}{2}\Delta_x +\mathbf V_1,\quad \mathbf H_y=-\frac{1}{2}\Delta_y +\mathbf V_2.
\end{equation*}
Then  $e^{-it\mathbf H_x}$ and
$e^{-it\mathbf H_y}$ enjoy Strichartz estimates, and \eqref{eq:system} can be
solved at the $L^2$ level, by a straightforward adaptation of \cite[Corollary~4.6.5]{CazCourant}:

\begin{theorem}\label{theo:CoulombStrichartz}
  Assume $d_1=d_2=3$, $\mathbf V_1,\mathbf V_2\in \mathcal Q$, $v_1,v_2,W\in
  L^p(\R^3)+L^\infty(\R^3)$ for some $p>3/2$,  and 
  $\phi_0^x,\phi_0^y\in L^2(\R^3)$. Then \eqref{eq:system} has a
  unique solution $(\phi^x,\phi^y)\in  C(\R;L^2(\R^3))^2\cap
  L^q_{\rm loc}(\R;L^r(\R^3))^2$,
where $1=2/r+1/p$ and $q$ is such that
\begin{equation*}
 \frac{2}{q}=3\(\frac{1}{2}-\frac{1}{r}\). 
\end{equation*}
The $L^2$-norms of $\phi^x$ and $\phi^y$ are independent of $t\in
\R$, hence in view of \eqref{eq:normL2},
\begin{equation*}
  \|\phi^x(t)\|_{L^2(\R^{3})}=\|\phi^y(t)\|_{L^2(\R^{3})}=1,\quad
  \forall t\in \R.
\end{equation*}
\end{theorem}

The proof is presented shortly in Appendix~\ref{sec:coulomb}.

\begin{remark}
The sign of $\eps$ in \eqref{eq:wCoulomb}  plays no role here. Indeed,
  the proof relies on local in time Strichartz estimates associated to
  $\mathbf H_x$ and $\mathbf H_y$, respectively, and the potentials
  $v_1,v_2$ and $W$ are treated as perturbations, whose sign is
  irrelevant in order to guarantee the above global existence
  result. On the other hand, Theorem~\ref{theo:CoulombStrichartz}
  brings no information regarding the quality of the dynamics or the existence of a ground state.
\end{remark}

\begin{remark}\label{rem:energy} Under extra assumptions on the potentials $v_1$ and $v_2$ (no extra
  assumption is needed for $W$, as it is associated to a convolution), it is
  possible to consider higher regularity properties. In particular,
  working at the level of $H^1$-regularity makes it possible to show
  the conservation of the energy
 \begin{align*}
  E(t)&=\langle H_x\phi^x(t),\phi^x(t)\rangle_{L^2_x}+
        \langle H_y\phi^y(t),\phi^y(t)\rangle_{L^2_y}\\
  &\quad
+  \iint_{\R^{3}\times\R^{3}}W(x-y)|\phi^x(t,x)|^2|\phi^y(t,y)|^2dxdy,
 \end{align*}
 provided that $\nabla v_1$ and $\nabla v_2$ also belong to
 $L^p(\R^3)+L^\infty(\R^3)$ for some $p>3/2$.
 We refer to Remark~\ref{rem:coulombH1} for more details.
 \end{remark}
 
 \begin{remark}
The role of the set $\mathcal Q$ is to guarantee that (local in
   time) Strichartz estimates are available for  $\mathbf H_x$ and
   $\mathbf H_y$. The same would still be true for a larger class of
   potentials, including for instance Kato potentials (\cite{RodnianskiSchlag}) or
   potentials decaying like an inverse square (\cite{BPST04}). The
   choice of this set $\mathcal Q$ is made in order to simplify the
   presentation, and because it is delicate to keep track of all the
   classes of potentials for which Strichartz estimates have been proved.
\end{remark}

\subsection{Coupling potentials with polynomial growth}\label{subsec:poly}
The core of this paper addresses the case where the coupling potential
$w$ may grow polynomially. To be more concrete, we recall the example 
addressed in \cite{BCFLL2}.

 \begin{example}\label{ex:companion} 
 Assume $d_1=d_2=1$ and that  the potentials are given by  
   \begin{equation*}
      V_1(x)= \frac 12 x^2 \left( \frac x{2\ell} -1\right)^2,\quad
      \ell>0,\quad  V_2(y)= \frac {\omega^2}2 y^2,\quad 
  w(x,y)  = \chi(x) y^2,\;\;\chi\in C_0^\infty(\R).
 \end{equation*}
 Here, $V_1(x)$ corresponds to a double well and $V_2(y)$ to a harmonic bath. The coupling $w(x,y)$ 
 could be locally cubic when choosing $\chi(x) = x$ for $x$ in a neighborhood of zero.
  \end{example}
  
We emphasize that in Example~\ref{ex:companion}, the average
$\<w\>_x$ grows quadratically in $y$: 
in terms of growth, $\<w\>_x$ is
comparable to~$V_2$ and cannot be considered as a perturbation as far as the Cauchy
problem is concerned. This setting  turns out to be very different from the one 
in~\cite{CancesLeBris99,Baudouin2005} (see also Theorem~\ref{theo:CoulombStrichartz}), 
and requires a different approach  to be developed below.

\subsubsection{Restriction to non-negative potentials}
\label{sec:hyp}
In the general case, we assume $d_1,d_2\ge 1$. First, 
 the potentials $V_1$ and $V_2$ are smooth, real-valued, $V_1\in
 C^\infty(\R^{d_1};\R)$, $V_2\in 
   C^\infty(\R^{d_2};\R)$, and bounded from below: 
\begin{equation*}
\forall x\in \R^{d_1}, \forall y\in \R^{d_2},\;\;
V_1(x)\ge -C_1\;\;\mbox{and}\;\;  V_2(y) \ge - C_2,
\end{equation*}
for some constants $C_1,C_2>0$. The operators $H_x$ and $H_y$ then are
self-adjoint operators. Up to changing $\phi^x(t,x)$ to
$\phi^x(t,x)e^{itC_1}$ (which amounts to replacing $V_1$ by
$V_1+C_1$ in \eqref{eq:system}), and  $\phi^y(t,y)$ to
$\phi^y(t,y)e^{itC_2}$, we may actually assume
\begin{equation*}
\tag{\textbf{H1}}
V_1(x)\ge 1,\ \forall x\in \R^{d_1}, \quad\text{and} \quad
V_2(y) \ge 1,\ \forall y\in \R^{d_2},
\end{equation*}
as we are only interested in existence results for the Cauchy problem
\eqref{eq:system}. 
Thus $H_x$ and $H_y$ are sums of a nonnegative operator (Laplacian in
$x$ and $y$, respectively) and of a nonnegative potential. We use them
to measure the regularity of the solutions of the
system~\eqref{eq:system}.

\subsubsection{Functional setting}
For $k\in\N$, we define the Hilbert spaces
\begin{align*}
\mathcal H_x^k & =\left\{ \phi\in L^2(\R^{d_1}),\;\; H^{k/2}_x \phi\in L^2(\R^{d_1})\right\} \;\;\mbox{and}\;\;
\mathcal H_y^k =\left\{ \phi\in L^2(\R^{d_2}),\;\; H^{k/ 2}_y \phi\in L^2(\R^{d_2})\right\} ,
\end{align*}
which are the natural analogues of Sobolev spaces $H^k$ in the
presence of (nonnegative) potentials (in view of {\bf (H1)}), equipped with the norms given by
\begin{equation*}
  \|\phi\|_{\mathcal H_x^k}^2 = \|\phi\|_{L^2_x}^2 + \| H^{k/2}_x
  \phi\|_{L^2_x}^2 ,\quad
\|\phi\|_{\mathcal H_y^k}^2 = \|\phi\|_{L^2_y}^2 + \| H^{k/2}_y
  \phi\|_{L^2_y}^2.
\end{equation*}
For $\alpha,\beta\in \N$,  $\Phi=(\phi^x,\phi^y) \in \mathcal
H_x^\alpha\times \mathcal 
H_y^\beta $, 
we set 
\[\| \Phi \|_{\alpha,\beta}^2=\|\phi^x\|_{\mathcal H_x^\alpha}^2 +
  \|\phi^y\|_{\mathcal H_y^\beta}^2 =
  \|\phi^x\|_{L^2_x}^2+ \|  H_x^{\alpha/2}
  \phi^x\|_{L^2_x}^2 + \|\phi^y\|_{L^2_y}^2+   \|  H_y^{\beta/2} \phi^y\|_{L^2_y}^2.\] 
  All along the paper, we use  that in view of {\bf (H1)},
  $0\le H_x^\alpha\le
H_x^{\alpha+1}$ and $0\le H_y^\beta\le
H_y^{\beta+1}$. 
\smallskip

As \eqref{eq:system} is reversible, from now on we consider positive
time only. 
We shall work with the time-dependent functional spaces
\[
X^{\alpha,\beta}_T= \left\{ \Phi(t)=(\phi^x(t),\phi^y(t)) ,\; \phi^x \in L^\infty\([0,T], \mathcal H_x^\alpha\),\;\phi^y \in L^\infty\([0,T], \mathcal H_y^\beta\)\right\}.
\]
If $\Phi=(\phi^x,\phi^y)\in X^{\alpha,\beta}_T$, we set 
\[
\| \Phi \|_{X_T^{\alpha,\beta}}=\sup_{t\in[0,T]} \| \Phi(t)\|_{\alpha,\beta}.
\] 
We choose to consider integer exponents 
 $\alpha$ and $\beta$ for the sake
of simplicity. We emphasize
however that our approach requires $\alpha,\beta\ge 2$; see
Section~\ref{sec:scheme} for a  more precise discussion on this
aspect.  We note that Theorem~\ref{theo:main}  
allows $\alpha=\beta=2$.

\subsubsection{Main assumptions}
We assume that the coupling potential $w\in C^\infty(\R^{d_1+d_2};\R)$
satisfies
\smallskip

\noindent \textbf{(H2) } 
There exist $c_0,C>0$ with $c_0<1$ such that for all
$(x,y)\in\R^{d_1}\times \R^{d_2}$,  
\begin{equation*}
  |w(x,y)| \le c_0( V_1(x) + V_2(y) +C).
\end{equation*}
We emphasize that  no condition is required
concerning the above constant~$C$: for instance if $V_1$
and $V_2$ are bounded, then we may always pick $c_0<1$ so that
{\bf (H2)} is satisfied. For unbounded potentials, the requirement
$c_0<1$ can be understood as some smallness property, in the sense
that $w(x,y)$ is a perturbation of $V_1(x)+V_2(y)$. This actually
corresponds to the physical framework where the system
\eqref{eq:system} is introduced in order to approximate the exact
solution $\psi$ of \eqref{eq:exact} through the formula $\psi\approx
\phi^x\otimes\phi^y$; see \cite{BCFLL21} for a derivation of error
estimates. We also note that the assumption $c_0<1$ implies that $w$
is $(H_x+H_y)$-bounded with relative bound $c_0<1$, hence by
Kato--Rellich Theorem (see
e.g. \cite[Theorem~X.12]{ReedSimon2}), $H$ is self-adjoint.

\begin{remark}
  If $w$ is at most quadratic, in the sense that $w\in C^\infty (\R^{d_1+d_2};\R)$
  \begin{equation*}
    \d_{x,y}^\gamma w\in L^\infty(\R^{d_1+d_2}),\quad \forall \gamma\in
    \N^{d_1+d_2},\ |\gamma|\ge 2,
  \end{equation*}
  then the assumption {\bf (H2)} is not needed to guarantee that $H$ is
  self-adjoint (see the Faris--Lavine Theorem,
  \cite[Theorem~X.38]{ReedSimon2}). Such  a framework corresponds to
  the assumptions made for the error analysis in \cite{BCFLL21,BCFLL2}. For such
  potentials $w$, the assumption $c_0<1$ in {\bf (H2)} is needed only in order to
  ensure that the  Hartree  solutions are global in time.
\end{remark}
\smallskip
\noindent We also assume some conditions on the regularity of 
commutators of the coupling potential with the operators $H_x$ and
$H_y$. For integers $\alpha,\beta \ge 1$, we consider the condition:
\smallskip

\noindent \textbf{(H3)}$_{\alpha,\beta}$.  
There exist $c_1,c_2>0$ such that for all $k\in\{1,\cdots, \alpha\}$, $\ell\in\{1,\cdots, \beta\}$, for all $f_j=f_j(x)$, $g_j=g_j(y)$  in
the Schwartz class ($j\in\{1,2\}$),
\begin{align*}
   \left| \< H_x^{k-1} [w(\cdot,y) , H_x] f_1, f_2\>_{L^2_x} \right| &
   +  \left| \<  [w(\cdot,y) , H_x]  H_x^{k-1} f_1, f_2\>_{L^2_x} \right|\\ 
   &  \le c_1\(1+
  V_2(y)\)\| f_1 \|_{\mathcal H_x^{k}}\|f_2  \|_{\mathcal H_x^{k}},
   \quad \text{for a.a. }y\in \R^{d_2},\\
   \left| \< H_y^{\ell-1}  [w(x,\cdot) , H_y] g_1,g_2\>_{L^2_y} \right| 
& + \left| \<  [w(x,\cdot) , H_y] H_y^{\ell-1} g_1,g_2\>_{L^2_y}  \right|\\
&  \le c_2\( 1+
  V_1(x) \)\|g_1\|_{\mathcal H_y^{\ell}} \|g_2\|_{\mathcal H_y^{\ell}} , 
   \quad \text{for a.a. }x\in \R^{d_1}.
\end{align*}
Assumption \textbf{(H3)}$_{\alpha,\beta}$ is made in order to
generalize the framework of 
 Example~\ref{ex:companion}. The subsequent proofs do not use the
 special form of the  
 Hamiltonians $H_x$, $H_y$, so that our result extends as soon as they
 are self-adjoint operators and assumptions \textbf{(H1)},
 \textbf{(H2)}, \textbf{(H3)}$_{\alpha,\beta}$ are satisfied. This
 applies in particular for magnetic Schr\"odinger operators.

\begin{remark}
  It is not necessary to assume that the potential $w$ is smooth,
  $w\in C^\infty(\R^{d_1+d_2};\R)$. We only need enough regularity in
  order to write assumption ${\bf (H3)}_{\alpha,\beta}$ for the $\alpha$ and
  $\beta$ that we consider (recalling that $H_x^{k-1}$ and
  $H_y^{\ell-1}$ are self-adjoint). For example, if $w\in
  C^2(\R^{d_1+d_2};\R)$, then Theorem~\ref{theo:main} holds with
    $\alpha=\beta=2$.  We make this regularity assumption for 
  simplicity, as the most important properties are those discussed in
  this subsection. 
\end{remark}

\begin{remark}
Whenever $w(x,y)$ is a Coulomb potential as in
 Section~\ref{subsec:coulomb}, the assumptions~\textbf{(H3)}$_{\alpha,\beta}$
 are not satisfied. One then needs to take advantage of the
 convolution feature of the coupling and of the properties of $H_x$
 and $H_y$ such as the Strichartz estimates in Proposition~\ref{prop:strichartz}.   
\end{remark}

We next present sufficient conditions on the potentials guaranteeing that assumptions
  {\bf (H2)} and {\bf(H3)}$_{2,2}$ hold.

\begin{lemma}[Sufficient conditions]\label{ex:extra}
Let $V_1\in C^\infty(\R^{d_1};\R)$ and $V_2\in C^\infty(\R^{d_2};\R)$
such that  $V_1,V_2\geq 0$.
  The above assumptions
  {\bf (H2)} and {\bf(H3)}$_{2,2}$  are satisfied provided that the
  following estimates hold:
  \begin{itemize}
  \item There exists $C>0$ such that
   \begin{equation}
     \label{eq:temperance}
     |\nabla V_1(x)|\le C\(1+V_1(x)\), \forall x\in \R^{d_1};\quad
     |\nabla V_2(y)|\le C\(1+V_2(y)\), \forall y\in \R^{d_2}.
   \end{equation}
   \item There exist  $0<c_0<1$ and $c>0$  independent of $x\in \R^{d_1}$
     and $y\in \R^{d_2}$ such that  
   \begin{equation}
     \label{hyp:stronger}
     \begin{cases}
  & |w(x,y)|\le c_0( V_1(x)+V_2(y)+c),\\
&|\nabla_x w(x,y)| \le c(\sqrt{V_1(x)}+V_2(y)+1),\\
  & |\nabla_y w(x,y)|\le c( V_1(x)+\sqrt{V_2(y)}+1),\\
  & |\Delta_x w(x,y)|+ |\Delta_y w(x,y)|\le c( V_1(x)+V_2(y)+1).
\end{cases}
 \end{equation}
  \end{itemize}
 \end{lemma}

The proof of this lemma is given in Section~\ref{sec:sufficient}.

\begin{remark}   
 We note that Example~\ref{ex:companion} meets the requirements stated in Lemma~\ref{ex:extra},  provided that  $\|\chi\|_{L^\infty(\R)}<\omega^2/2$. Thus it satisfies the assumptions of Theorem~\ref{theo:main} below for $\alpha=\beta=2$.
\end{remark}

\subsection{Main result and comments}
\label{sec:main-result}
Before stating our main result we informally summarize the previous assumptions on the potentials for the case of polynomial coupling:
\begin{align*}
&{\bf (H1)}: \text{boundedness from below of the potentials $V_1(x)$ and $V_2(y)$;}\\
&{\bf (H2)}: \text{control of $w(x,y)$ in terms of $V_1(x)+V_2(y)$;}\\
&{\bf (H3)}_{\alpha,\beta}: \text{control of commutators involving
 $w(x,y)$, in terms of $H_x$ and $H_y$.} 
\end{align*}
We have the following result on existence and uniqueness as well norm
and energy conservation of the time-dependent Hartree approximation.

\begin{theorem}\label{theo:main}
Let $d_1,d_2\ge 1$, $\alpha,\beta \ge 2$ and $\phi_0^x\in\mathcal H_x^
\alpha$, $\phi_0^y\in \mathcal H_y^\beta$.  Suppose that {\bf
  (H1)}, {\bf (H2)} and~{\bf (H3)}$_{\alpha,\beta}$ are satisfied. 
\begin{itemize}
\item \eqref{eq:system} possesses a unique, global solution in $\Phi\in
C(\R_+;L^2\times 
 L^2)\cap \bigcap_{T>0} X^{\alpha,\beta}_T$. 
\item \emph{Conservations:} the $L^2$-norms of $\phi^x$ and $\phi^y$
  are independent of $t\ge 0$, hence in view of \eqref{eq:normL2},
\begin{equation*}
   \|\phi^x(t)\|_{L^2(\R^{d_1})} =  \|\phi^y(t)\|_{L^2(\R^{d_2})}
   =1,\quad \forall t\ge 0.
 \end{equation*}
 In addition, the following total energy is also independent of $t\ge 0$:
 \begin{align*}
  E(t)&:=\langle H_x\phi^x(t),\phi^x(t)\rangle_{L^2_x}+
        \langle H_y\phi^y(t),\phi^y(t)\rangle_{L^2_y}\\
  &\quad+
  \iint_{\R^{d_1}\times\R^{d_2}}w(x,y)|\phi^x(t,x)|^2|\phi^y(t,y)|^2dxdy.
\end{align*}
\end{itemize}
\end{theorem}

We will see that the assumption $c_0<1$  in {\bf (H2)} arises in two steps of the proof
of Theorem~\ref{theo:main}. First, to make sure that the approximating
scheme \eqref{eq:scheme} introduced below is well-defined, we invoke
Kato--Rellich Theorem, to show essentially that $\<w\>_y$ (or more
precisely, $\<w_n\>_y$) is
$H_x$-bounded with relative bound smaller than one, and that the same
holds when the roles of $x$ and $y$ are swapped. Second, the
assumption $c_0<1$ guarantees that the conserved energy $E$, defined in
Theorem~\ref{theo:main}, is a coercive functional, so the conservation
of $E$ provides uniform in times a priori estimates, which in turn
allow to show that the local in time solutions are actually global in
time solutions. 

The property $\Phi\in
C(\R_+;L^2\times  L^2)\cap \bigcap_{T>0} X^{\alpha,\beta}_T$ means
that $t\mapsto \|\Phi(t)\|_{\alpha,\beta}$ is \emph{locally} bounded
on $\R_+$. The map  $t\mapsto \|\Phi(t)\|_{1,1}$ is bounded on $\R_+$ in view of the
conservation of the coercive energy $E$, but higher order norms may
not be bounded as $t$ goes to infinity (recall that
Theorem~\ref{theo:main} requires $\alpha,\beta\ge 2$).

\section{Analysis of the iterative scheme: existence and uniform
  bounds}  
\label{sec:scheme}

This section is devoted to the analysis of the system~\eqref{eq:scheme}. For $n\in\N$, we denote by
 $\Phi_n=(\phi_n^x,\phi_n^y)$, the solution to the scheme~\eqref{eq:scheme} and we prove local in time uniform estimates. 
 At this stage, we only need that $\phi_0^x\in\mathcal H_x^ \alpha$, $\phi_0^y\in \mathcal
H_y^\beta$ for
 integers $\alpha,\beta \ge 1$.

\begin{lemma}\label{lem:unif}
Let $\alpha,\beta\ge 1$. Assume that {\bf (H1)}, {\bf (H2)} and
{\bf (H3)}$_{\alpha,\beta}$ are satisfied. 
 Assume  $\phi_0^x\in  \mathcal
   H_x^\alpha$, and $\phi_0^y\in \mathcal 
  H_y^\beta$. Then, the sequence $(\Phi_n)_{n\in\N}$ solution to~\eqref{eq:scheme} is well-defined and there 
  exists $T>0$ such that for all $n\in \N$, the solution $\Phi_n\in
  X_T^{\alpha,\beta}$ of the scheme~\eqref{eq:scheme} satisfies 
    \begin{equation}\label{eq:Phinab}
    \|\Phi_n\|_{X_T^{\alpha,\beta}}\le 2\|\Phi_0\|_{\alpha,\beta}. 
  \end{equation}
\end{lemma}

The proof of this lemma relies on the fact that the control of  $\Phi_{n+1}$
involves terms which are linear in $\Phi_{n+1}$ and quadratic in
$\Phi_{n}$, and require the $X_T^{1,1}$-norm of $\Phi_{n}$. 
For this reason, the Lemma holds as soon as $\alpha,\beta \geq 1$. However in
Theorem~\ref{theo:main} , we require at least an $X^{2,2}_T$
regularity.  The reason will appear in
Section~\ref{sec:convergence}, as we do need uniform (in $n$)
estimates in $X^{2,2}_T$ to show that the sequence $(\Phi_n)_n$
converges in $X^{0,0}_T=L^\infty_T L^2$. 

In Section~\ref{sec:prelim}, we address the construction of the family $(\Phi_n)_{n\in\N}$, which relies on a commutation lemma that we prove in Section~\ref{sec:proof-lemma}. Section~\ref{sec:unif} is devoted to the proof of the uniform bound stated in Lemma~\ref{lem:unif}.

\subsection{Well-posedness of the scheme}\label{sec:prelim}

Before entering into the proof of Lemma~\ref{lem:unif}, let us discuss 
why the scheme is indeed well-defined: as $\Phi_{n+1}$ solves a
decoupled system of linear Schr\"odinger equations, it suffices to
study the properties of the time-dependent potentials $\<w_n\>_y$ and
$\<w_n\>_x$.  
 We fix $T>0$ arbitrary and take $\phi_0^x\in  \mathcal
   H_x^\alpha$, and $\phi_0^y\in \mathcal 
  H_y^\beta$, with $\alpha,\beta\ge 1$. 
For $n=0$, $\Phi_0$ is obviously well-defined with $\Phi_0\in
X_T^{\alpha,\beta}$, and
\begin{equation}
  \label{eq:recL2}
\|\phi_n^x(t)\|_{L^2_x}= \|\phi_n^y(t)\|_{L^2_y}=1 ,\quad \forall t\in \R,
\end{equation}
holds for $n=0$. We argue by induction. If $\Phi_n\in X_T^{1,1}$ satisfies \eqref{eq:recL2},
then in view of {\bf (H2)}, $\langle w_n\rangle_y(t,x)$ and
$\langle w_n\rangle_x(t,y)$ are well-defined. In addition, for $t\in
[0,T]$, {\bf (H2)} yields
\begin{equation}\label{eq:est-wn}
\begin{aligned}
 & |\langle w_n\rangle_y(t,x)|\le
 c_0 V_1(x)\|\phi^y_n(t)\|_{L^2}^2
+C \|\phi_n^y(t)\|^2_{ \mathcal H_y^{1}}
,\quad \text{ a.e. }x,\\
& |\langle w_n\rangle_x(t,y)|\le
 c_0 V_2(y) \|\phi^x_n(t)\|^2_{L^2} +C\|\phi_n^x(t)\|^2_{\mathcal H_x^{1}},
\quad \text{a.e. } y,
\end{aligned}
\end{equation}
for some constant $C$ whose value is irrelevant here, unlike the fact
that we assume $c_0<1$.
Indeed, together with \eqref{eq:recL2},  this implies that $\langle w_n\rangle_y$ is $H_x$-bounded with
relative bound at most $c_0$. By Kato--Rellich Theorem (see
e.g. \cite[Theorem~X.12]{ReedSimon2}),  $\Phi_{n+1} \in X_T^{0,0}$
is well-defined
(see e.g. \cite[Section~VIII.4]{ReedSimon1}), and \eqref{eq:recL2}
holds at level $n+1$.
 Next, we prove that
 $\Phi_{n+1}\in X_T^{1,1}$. 
 Applying the  operator $H_x$ to the
first equation in \eqref{eq:scheme}, we find
\begin{equation}\label{eq:tata3}
  (i\d_t -H_x)(H_x \phi_{n+1}^x) =\langle w_n\rangle_y (t)(H_x \phi_{n+1}^x )+ 
  [ H_x, \langle w_n\rangle_y(t)] \phi^x_{n+1}.
\end{equation}
Since $H_x$ is self-adjoint,  we deduce 
 \begin{align*}
  \| H^{1/2}_x \phi^x_{n+1}(t) \|_{L^2_x}^2& = {\rm Re}  \< H_x \phi^x_{n+1}(t),\phi^x_{n+1}(t)\>_{L^2_x}\\
 &=   \| H^{1/2}_x \phi^x_0 \|_{L^2_x}^2+{\rm Re}  \left(\int_0^t \frac d{ds} \< H_x \phi^x_{n+1}(s),\phi^x_{n+1}(s)\>_{L^2_x}ds\right)\\
 & = \| H^{1/2}_x \phi^x_0 \|_{L^2_x}^2- {\rm Re} \left( \int_0^t \<i  [ H_x, \langle w_n\rangle_y] \phi^x_{n+1} (s) , \phi^x_{n+1} (s) \>_{L^2_x} ds\right).
 \end{align*}
 Minkowski inequality yields, in view of {\bf (H3)}$_{1,1}$,
\begin{equation}\label{eq:com2}
    \left|\< [H_x,\langle w_n\rangle_y(t)] f_1, f_2\>_{L^2_x}\right|  \lesssim \|
    \phi_n^y\|_{L^\infty_T\mathcal H^1_y}^2   \| f_1\|_{\mathcal H_x^1}  \| f_2\|_{\mathcal H_x^1}.
\end{equation}
We infer  the existence of a universal constant $C>0$ such that
 \begin{align*}
  \| H^{1/2}_x \phi^x_{n+1}(t) \|_{L^2_x}^2&
 \leq \| H^{1/2}_x \phi^x_0 \|_{L^2_x}^2+ C\|\Phi_{n}\|_{X^{1,1}_T}^2
   \int_0^t \| \phi^x_{n+1}(s)\|^2_{\mathcal H_x^1} ds.
\end{align*}
We deduce 
\begin{equation}\label{eq:tata1}
\sup_{t\in[0,T]}  \| \phi^x_{n+1}(t)\|^2_{\mathcal H_x^1} \lesssim  \| \phi^x_{0}\|^2_{\mathcal H_x^1}  + C
\|\Phi_{n}\|_{X^{1,1}_T}^2
   \int_0^t \| \phi^x_{n+1}(s)\|^2_{\mathcal H_x^1} ds,
\end{equation}
 We have a similar estimate for $\| H^{1/2}_y \phi^y_{n+1}(t)
 \|_{L^2_x}^2$:
 \begin{equation}\label{eq:tata2}
\sup_{t\in[0,T]}  \| \phi^y_{n+1}(t)\|^2_{\mathcal H_y^1} \lesssim  \| \phi^y_{0}\|^2_{\mathcal H_y^1}  + C
\|\Phi_{n}\|_{X^{1,1}_T}^2
   \int_0^t \| \phi^y_{n+1}(s)\|^2_{\mathcal H_y^1} ds,
\end{equation}
and so Gronwall lemma and the inductive assumption yields $\Phi_{n+1}\in
 X_T^{1,1}$ and completes the construction of the sequence $(\Phi_n)_{n\in\N}$. 

\subsection{Uniform bounds}
\label{sec:unif}

We conclude the proof of Lemma~\ref{lem:unif} in analyzing the regularity of the solutions. 
%

\begin{proof}[Proof of Lemma~\ref{lem:unif}]
  In view of the definition of the scheme and of the conservations
  \begin{equation*}
    \frac{d}{dt}\|\phi_n^x\|_{L^2}^2 = \frac{d}{dt}\|\phi_n^y\|_{L^2}^2=0,
  \end{equation*}
we need now consider $H_x^\alpha \phi^x_n$ and $H_y^\beta\phi^y_n$ for $\alpha,\beta\geq 1$.
 Let
\begin{equation*}
  R = 2\|\Phi_0\|_{\alpha,\beta},
\end{equation*}
and introduce
\begin{equation*}
  B_{R,T}=\{\Phi\in X_T^{\alpha,\beta},\ \|\Phi\|_{X_T^{\alpha,\beta}}\le
  R\}. 
\end{equation*}
 We distinguish two cases for the ease of presentation. 
 \medskip

\noindent{\bf First case: $\alpha=\beta=1$.}
In that case, if $\Phi_n\in B_{R,T}$, then estimates \eqref{eq:tata1} and~\eqref{eq:tata2} imply 
\begin{equation*}
   \|  \Phi_{n+1}(t)\| _{X_T^{1,1}}^2  \le \| 
  \Phi_{0}\| _{1,1}^2+ C TR^2\|\Phi_{n+1}(t)\|
  _{X_T^{1,1}}^2. 
\end{equation*}
We infer that choosing $T>0$ sufficiently small in terms of $R$, but
independently of $n$, $\Phi_n\in
B_{R,T}$ implies $\Phi_{n+1}\in B_{R,T}$.
\medskip

\noindent{\bf Higher regularity}:
 The control of higher
 order regularity is obtained  by a similar recursive argument which uses an
 iterated commutator estimate. Let $\alpha,\beta\geq 1$. 
 We have 
  \begin{equation}
\label{eq:scheme_alpha}
\left\{
\begin{aligned}
i\partial_t H_x^k \phi^x_{n+1}= (H_x  + \langle w_n\rangle_y )H_x^k \phi^x_{n+1} 
+ [ H^k_x, \langle w_n\rangle_y] \phi^x_{n+1} 
,\quad H_x^k\phi^x_{n+1\mid t=0}= H_x^k\phi^x_0,\\
i\partial_t  H_y^\ell  \phi^y_{n+1}=( H_y  + \langle w_n\rangle_x)H_y^\beta
\phi^y_{n+1} 
+[H_y^\ell,\langle w_n\rangle_x] \phi^y_{n+1}
,\quad H_y^\ell\phi^y_{n+1\mid t=0}= H_y^\ell\phi^y_0.
\end{aligned}
\right.
\end{equation}
 The next lemma allows to control the commutators.  
 
 \begin{lemma}\label{lem:recursive}
 Let $ \Phi_n\in X^{1,1}_T$ for some $T>0$, and $k,\ell \ge 1$ be
 integers. Suppose that {\bf (H3)}$_{k,\ell}$ is satisfied.
 For all $t\in [0,T]$, 
 \begin{align*}
 \left|\< [H_x^k, \langle w_n\rangle_y(t)] f_1,
   f_2\>_{L^2_x}\right|& \lesssim \| \phi_n^y\|^2_{L^\infty_T
   \mathcal H^{1} _y} 
\| f_1\|_{\mathcal H^k_x} \|f_2\| _{\mathcal H^k_x} ,
\qquad \forall f_1,f_2\in \mathcal H_x^{k}\\
 \left| \< [H_y^\ell, \langle w_n\rangle_x(t)]
   g_1,g_2\>_{L^2_y}\right|
& \lesssim \| \phi_n^x\|^2_{L^\infty_T
 \mathcal H^1_x} 
\| g_1\|_{\mathcal H^\ell_y} \| g_2\| _{\mathcal H^\ell_y},
\qquad \forall g_1,g_2\in \mathcal H_y^{\ell} . 
 \end{align*}
 \end{lemma}

 Taking the lemma for granted,  \eqref{eq:scheme_alpha} implies, since $H_x$ is self-adjoint,  
 \begin{align*}
 \|\phi_{n+1}^x(t)\|_{\mathcal H^\alpha _x}^2 &
  = {\rm Re}  \< H_x ^\alpha \phi^x_{n+1}(t),\phi^x_{n+1}(t)\>_{L^2_x}\\
 &=   \| H^{\alpha/2}_x \phi^x_0 \|_{L^2_x}^2+{\rm Re}  \left(\int_0^t \frac d{ds} \<H_x ^\alpha \phi^x_{n+1}(s),\phi^x_{n+1}(s)\>_{L^2_x}ds\right)\\
 & = \| H^{\alpha/2}_x \phi^x_0 \|_{L^2_x}^2- {\rm Re} \left( \int_0^t \<i  [ H_x^\alpha , \langle w_n\rangle_y] \phi^x_{n+1} (s) , \phi^x_{n+1} (s) \>_{L^2_x} ds\right)\\
 & \le \| H_x ^{\alpha/2} \phi_{0}^x\| _{L^2_x}^2+ C T \|\phi^y_n\|^2_{L^\infty_T \mathcal
   H^1_y} \sup_{t\in[0,T]} \|\phi_{n+1}^x(t)\|_{\mathcal H^\alpha _x}^2.
 \end{align*}

 We have a similar estimate for $\| H^{\beta/2}_y \phi^y_{n+1}(t)
 \|_{L^2_x}^2$, and so 
   if $\Phi_n\in B_{R,T}$, then equations~\eqref{eq:tata1}
   and~\eqref{eq:tata2} imply
   \begin{equation*}
   \|  \Phi_{n+1}(t)\| _{X_T^{\alpha,\beta}}^2  \le \|
  \Phi_{0}\| _{\alpha,\beta}^2+ C TR^2\|\Phi_{n+1}(t)\|
  _{X_T^{\alpha,\beta}}^2. 
\end{equation*}
We infer that choosing $T>0$ sufficiently small in terms of $R$, but
independently of $n$, $\Phi_n\in
B_{R,T}$ implies $\Phi_{n+1}\in B_{R,T}$.
It thus remains to prove the lemma, which is the subject of the next subsection. 
\end{proof}

\begin{remark}
 Lemma~\ref{lem:unif} holds as soon as $\alpha,\beta\ge 1$, but this
 is not enough in order to conclude that the sequence $(\Phi_n)_{n\in
   \N}$ converges to some solution of \eqref{eq:system}. Indeed, the
 mere boundedness in $X_T^{1,1}$ only implies the convergence of a subsequence in the
 weak-* topology: this is not enough to pass to the limit in
 \eqref{eq:scheme}, both because the subsequence need not retain
 consecutive indices, and because the topology considered is too large
 to pass to the limit in nonlinear terms. These issues are
 overcome by requiring $\alpha,\beta\ge 2$ in
 Sections~\ref{sec:convergence} and~\ref{sec:solution}. 
\end{remark}

\subsection{Proof of Lemma~\ref{lem:recursive}}
\label{sec:proof-lemma}

Of course,~\eqref{eq:com2} implies the result
 when $k=1$. Take $k\geq 1$ and 
assume that the result holds for all $m \leq k$.
We write 
\begin{align*}
[ H_x^{k+1}, \langle w_n\rangle_y(t)] & = H_x^k[H_x,  \langle w_n\rangle_y(t)] + [H_x^k, \langle w_n\rangle_y(t)] H_x\\
&= H_x[H^{k-1}, \langle w_n\rangle_y(t)]H_x + H_x^k [H_x, \langle w_n\rangle_y(t)]  + [H_x , \langle w_n\rangle_y(t)] H_x^k. 
\end{align*}
We deduce from~\eqref{eq:com2} and the recursive assumption that for $f_1,f_2\in\mathcal \Sch(\R^{d_1})$, we have 
\begin{align*}
\left|\<[ H_x^{k+1}, \< w_n\>_y(t)]f_1,f_2\>_{L^2_x} \right|&\le
\left|\<[ H_x^{k-1}, \< w_n\>_y(t)] H_xf_1,H_xf_2\>_{L^2_x} \right|\\
&\quad +\left|\< H_x^k [ H_x, \< w_n\>_y(t)]f_1,  f_2\>_{L^2_x}\right|+\left|\<[ H_x, \<
  w_n\>_y(t)]  H_x^k  f,g\>_{L^2_x}\right|.
  \end{align*}
  By  the recursive assumption 
  \begin{align*}
  \left|\<[ H_x^{k-1}, \< w_n\>_y(t)] H_xf_1,H_xf_2\>_{L^2_x} \right| &  \lesssim 
\| \phi^y_n\|_{L^\infty_T\mathcal H_y^1}^2 \| H_x f_1\|_{\mathcal
   H^{k-1}_x } \| H_x f_2\|_{\mathcal H^{k-1}_x}  
\\
&\lesssim
 \| \phi^y_n\|_{L^\infty_T\mathcal H_y^1}^2 \| f_1\|_{\mathcal H^{k+1}_x } \| f_2\|_{\mathcal H^{k+1}_x} .
\end{align*}
Finally,  in view of {\bf (H3)}$_{k,\ell}$ and Minkowski
inequality, we have 
\begin{align*}
 \left|   \< H_x^k [H_x,\< w_n\>_y(t)] f_1, f_2\>_{L^2_x} \right|   & \lesssim \|
    \phi_n^y\|_{L^\infty_T\mathcal H^1_y}^2   \| f_1\|_{\mathcal
    H_x^{k+1}}  \| f_2\|_{\mathcal H_x^{k+1}},\\ 
 \left|   \<  [H_x,\< w_n\>_y(t)]  H_x^k f_1, f_2\>_{L^2_x}  \right| 
 & \lesssim \|
    \phi_n^y\|_{L^\infty_T\mathcal H^1_y}^2   \| f_1\|_{\mathcal H_x^{k+1}}  \| f_2\|_{\mathcal H_x^{k+1}},
    \end{align*}
which concludes the proof, after arguing similarly with $H_y$. 


\section{Convergence in small norms}
\label{sec:convergence}

The second step of the proof of Theorem~\ref{theo:main} consists in passing to the limit $n\rightarrow +\infty$ and prove the existence of a limit to the sequence $(\Phi_n)_{n\in\N}$ of solutions to~\eqref{eq:scheme}. The main result in this section is:
\begin{lemma}\label{lem:CVL2}
 Assume that there exist $T>0$ and $R>0$ such that 
 \begin{equation*}
   \sup_{n\in \N}\|\Phi_n\|_{X_T^{2,2}}\le R.
 \end{equation*}
Then there exist $T_1\in ]0,T]$ and  $\Phi \in X_{T_1}^{2,2}$ such that 
\begin{equation}\label{eq:CVL2}
\sup_{0\le t\le T_1}\|\Phi_n(t)-\Phi(t)\|_{L^2_x\times
  L^2_y}=\|\Phi_n-\Phi\|_{X^{0,0}_{T_1}}\Tend n \infty 0. 
\end{equation}
If in addition $(\Phi_n)_n$ is bounded in $X_T^{\alpha,\beta}$ for
some integers $\alpha,\beta\ge 2$, then $\Phi \in X_{T_1}^{\alpha,\beta}$.
\end{lemma}

\begin{proof}
  Consider \eqref{eq:scheme} at steps $n+1$ and $n$, respectively, and
  subtract the corresponding equations. We find, for $n\ge 1$,
  \begin{align*}
    \(i\d_t -H_x\) \(\phi_{n+1}^x-\phi_n^x\) & = \<w_n\>_y\phi_{n+1}^x
- \<w_{n-1}\>_y\phi_{n}^x\\
&=\<w_n\>_y\(\phi_{n+1}^x-\phi_{n}^x\) +\(
                              \<w_n\>_y-\<w_{n-1}\>_y\)\phi_{n}^x,
  \end{align*}
and  energy estimates yield, for $T_1\in ]0,T]$, since $\Phi_{n+1\mid
  t=0}=\Phi_{n\mid t=0}$,
\begin{equation}\label{eq:CVL2-1}
  \|\phi_{n+1}^x-\phi_n^x\|_{L^\infty_{T_1}L^2_x}\le \int_0^{T_1}\left\|\(
                              \<w_n(s)\>_y-\<w_{n-1}(s)\>_y\)\phi_{n}^x(s)\right\|_{L^2_x}ds.
\end{equation}
In view of {\bf (H2)}, the key term is estimated by
\begin{equation*}
\left|  \<w_n(t)\>_y-\<w_{n-1}(t)\>_y\right| \lesssim
\int_{\R^{d_2}} (V_1(x)+V_2(y)+1) \left|
  |\phi_n^y(t,y)|^2-|\phi^y_{n-1}(t,y)|^2\right| dy .
\end{equation*}
Writing $ |\phi_n^y|^2-|\phi^y_{n-1}|^2=
\Re\((\phi_n^y-\phi^y_{n-1})(\overline{\phi_n^y}+\overline{\phi^y_{n-1}})\)$,
and using Cauchy-Schwarz inequality,
\begin{align}
\nonumber
\left|  \<w_n(t)\>_y-\<w_{n-1}(t)\>_y\right| &\lesssim
(V_1(x)+1) \(\|\phi_n^y\|_{L^2_y} +
  \|\phi_{n-1}^y\|_{L^2_y}\)\|\phi_n^y- \phi_{n-1}^y\|_{L^2_y} \\
  \nonumber
  &\quad+
\(\|V_2\phi_n^y\|_{L^2_y} + 
    \|V_2\phi_{n-1}^y\|_{L^2_y}\)\|\phi_n^y- \phi_{n-1}^y\|_{L^2_y} \\
    \label{toto}
  &\lesssim \(V_1(x)+1)\)\sup_{k\in
    \N}\|\Phi_k\|_{X^{2,2}_T}\|\phi_n^y- \phi_{n-1}^y\|_{L^2_y} . 
\end{align}
Plugging this estimate into \eqref{eq:CVL2-1}, we infer, thanks to
Minkowski inequality,
\begin{align*}
   \|\phi_{n+1}^x-\phi_n^x\|_{L^\infty_{T_1}L^2_x}&\lesssim \sup_{k\in
    \N}\|\Phi_k\|_{X^{2,2}_T}
  \int_0^{T_1}\left\|\(V_1+1)\)\phi_{n}^x(s)\right\|_{L^2_x}\|\phi_n^y(s)-
  \phi_{n-1}^y(s)\|_{L^2_y}  ds\\
  &\lesssim \sup_{k\in
    \N}\|\Phi_k\|_{X^{2,2}_T}^2
  \int_0^{T_1}\|\phi_n^y(s)-\phi_{n-1}^y(s)\|_{L^2_y}  ds\\
  &\lesssim R^2 T_1 \sup_{t\in [0,T_1]} \|\phi_n^y(t)-\phi_{n-1}^y(t)\|_{L^2_y}  .
\end{align*}
We obtain a similar estimate by exchanging the roles of $x$ and $y$,
and so
\begin{equation}\label{eq:contractL2}
  \|\Phi_{n+1}-\Phi_n\|_{X^{0,0}_{T_1}}\lesssim R^2 T_1
  \|\Phi_{n}-\Phi_{n-1}\|_{X^{0,0}_{T_1}}. 
\end{equation}
Fixing $T_1\in ]0,T]$ sufficiently small, the series
\begin{equation*}
  \sum_{n\in \N} \|\Phi_{n+1}-\Phi_n\|_{X^{0,0}_{T_1}}
\end{equation*}
converges geometrically, and $\Phi_n$ converges in $X_{T_1}^{0,0}$, to
some $\Phi \in X_{T_1}^{0,0}$. 

On the other hand, the boundedness of $(\Phi_n)_n$ in $X_{T}^{2,2}$
implies that a subsequence is converging in the weak-* topology of
$X_{T}^{2,2}$. By uniqueness of limits in the sense of distributions,
we infer $\Phi \in X_{T_1}^{2,2}$. The same holds when $X_{T}^{2,2}$
is replaced by $X_T^{\alpha,\beta}$ for $\alpha,\beta\ge 2$. 
\end{proof}


\section{Passing to the limit in the equation}
\label{sec:solution}

We now have all the elements in hands for proving Theorem~\ref{theo:main} by showing that the limit  function $\Phi$ constructed in Lemma~\ref{lem:CVL2} is a solution to equation~\eqref{eq:system} with the properties stated in  Theorem~\ref{theo:main}.

\subsection{Existence of a local solution}
Combining Lemmas~\ref{lem:unif} and \ref{lem:CVL2}, we infer that
under the assumptions of Theorem~\ref{theo:main}, there exists $T_1>0$
such that $\Phi_n\to \Phi$ in $X_{T_1}^{0,0}$. By uniqueness of the
limit, we also have $\Phi_n\rightharpoonup \Phi$ in $X_{T_1}^{\alpha,\beta}$
(and no extraction of a subsequence is needed).
 Resuming the estimates from the proof of Lemma~\ref{lem:CVL2},
 we observe that for $n,m\in\N$, $t\in [0,T_1]$ and $x\in\R^{d_1}$,
  \begin{align*}
  |\<w_n(t)\>_y-\<w_m(t)\>_y|&= \left|\int_{\R^{d_2}}w(x,y)\(
                       |\phi_n^y(t,y)|^2-|\phi_m^y(t,y)|^2\)dy\right| \\
   &\lesssim \(V_1(x)+1\)\sup_{k\in
     \N} \|\Phi_k\|_{X^{2,2}_T}\|\phi_n^y(t)-\phi_m^y(t)\|_{L^2_y}.
 \end{align*}
Passing to the limit $m\rightarrow +\infty$, we obtain that 
 for $n\in\N$, $t\in [0,T_1]$ and $x\in\R^{d_1}$,
 \begin{align*}
  |\<w_n(t)\>_y-\<w(t)\>_y|
   &\lesssim \(V_1(x)+1\)\sup_{k\in
     \N} \|\Phi_k\|_{X^{2,2}_T}\|\phi_n^y(t)-\phi^y(t)\|_{L^2_y}.
 \end{align*}
 Therefore, keeping the same notation $R$ as from Lemma~\ref{lem:CVL2},
 \begin{align*}
   \left\| \<w_n(t)\>_y\phi_{n+1}^x(t) -
   \<w(t)\>_y\phi^x(t)\right\|_{L^2_x}&\lesssim
   \left\| \(\<w_n(t)\>_y-\<w(t)\>_y\)\phi_{n+1}^x(t)\right\|_{L^2_x}\\
&\quad+ \left\| \<w(t)\>_y\(\phi_{n+1}^x(t)
  -\phi^x(t)\)\right\|_{L^2_x}\\
 &\lesssim
  R \|\phi_n^y(t)-\phi^y(t)\|_{L^2_y}\left\| (V_1+1)\phi_{n+1}^x(t)\right\|_{L^2_x}\\
&\quad+ \left\| \(V_1+1+\|\phi^y\|^2_{L^\infty_{T_1}\mathcal H^1_y}\)\(\phi_{n+1}^x(t)
  -\phi^x(t)\)\right\|_{L^2_x},
 \end{align*}
 where we have used \eqref{eq:est-wn} and the normalization
 \eqref{eq:normL2}. The first term on the right hand side goes to zero
 as $n\to \infty$, uniformly in $t\in [0,T_1]$. So does the last one
 in the case $\alpha,\beta\ge 3$, since by interpolation $\Phi_n$ then
 converges to $\Phi$ strongly in $X^{2,2}_T$. In the case where
 $\alpha$ or $\beta$ is equal to $2$, we can only claim a weak
 convergence, 
\begin{equation*}
   \<w_n\>_y \phi_{n+1}^x\Tendweak n \infty \<w\>_y \phi^x\quad\text{in
  }L^\infty([0,T_1];L^2_x) \text{ weak-$*$}.
  \end{equation*}
  Similarly,
  \begin{equation*}
  \<w_n\>_x \phi_{n+1}^y\Tendweak n \infty \<w\>_x \phi^y\quad\text{in
    }L^\infty([0,T_1];L^2_y) \text{ weak-$*$},
\end{equation*}
and $\Phi $ solves \eqref{eq:system} for $t\in [0,T_1]$, in the sense
of distributions. In view of the regularity $\Phi\in X_{T_1}^{\alpha,\beta}$, Duhamel's
formula,
\begin{align*}
  & \phi^x(t) = e^{-itH_x}\phi_0^x -i\int_0^t e^{-i(t-s)H_x} \(
    \<w\>_y\phi_x\)(s)ds,\\
& \phi^y(t) = e^{-itH_y}\phi_0^y -i\int_0^t e^{-i(t-s)H_y} \(
    \<w\>_x\phi_y\)(s)ds,
\end{align*}
then shows the continuity in time $\Phi\in C([0,T_1];L^2_x\times L^2_y)$. 

\subsection{Uniqueness}
\label{sec:uniqueness}

At this stage, it is rather clear that uniqueness holds in
$X^{2,2}_T$, no matter how large $\alpha$ and $\beta$ are. 
Suppose that $\tilde\Phi \in  X_{T}^{2,2}$ is another solution to
\eqref{eq:system} for $T>0$: the system satisfied by $\Phi-\tilde
\Phi$ is similar to the one satisfied by $\Phi_{n+1}-\Phi_n$, and
considered in the proof of Lemma~\ref{lem:CVL2}. Since
$\Phi,\tilde\Phi \in  X_{T}^{2,2}$, there exists $R>0$ such that 
\begin{equation*}
  \|\Phi\|_{X_{T}^{2,2}}+\|\tilde\Phi\|_{X_{T}^{2,2}}\le R, 
\end{equation*}
and repeating the computations presented in the proof of
Lemma~\ref{lem:CVL2}, we obtain, for any $T_1\in ]0,T]$, 
\begin{equation*}
    \|\Phi-\tilde \Phi\|_{X_{T_1}^{0,0}}\le
  CT_1R \|\Phi-\tilde\Phi\|_{X_{T_1}^{0,0}}.
\end{equation*}
Picking $T_1>0$ such that $CT_1R <1$ shows that $\Phi\equiv \tilde
\Phi$ for $t\in [0,T_1]$, and we infer that $\Phi\equiv \tilde \Phi$
on $[0,T]$ by covering $[0,T]$ by finitely many intervals of length at
most $T_1$.

\subsection{Conservations}
\label{sec:conservations}

We now address the second point in Theorem~\ref{theo:main}: we assume
that \eqref{eq:system} has a unique solution $\Phi\in X^{2,2}_T$ for
some $T>0$. This implies in particular, in view of~\eqref{eq:system},
that $\d_t\phi^x\in L^\infty([0,T];L^2_x)$ and $\d_t\phi^y\in
L^\infty([0,T];L^2_y)$, and the multiplier techniques evoked below are
justified without using regularizing argument as in
e.g. \cite{CazCourant}.

For the conservation of the $L^2$-norms, multiply the first equation in
\eqref{eq:system} by $\overline{\phi^x}$, integrate in space on
$\R^{d_1}$, and consider the imaginary part: we readily obtain
\begin{equation*}
  \frac{d}{dt}\|\phi^x(t)\|_{L^2_x}^2=0.
\end{equation*}
We proceed similarly for $\phi^y$, and the conservation of the
$L^2$-norms follows.

For the energy, consider the multiplier $\d_t \overline{\phi^x}$ in
the equation for $\phi^x$: as evoked above, all the products are
well-defined, in the worst possible case as products of two $L^2$
functions. Integrate in space and consider the real part: we obtain
\begin{equation*}
  \frac{d}{dt}E(t)=0.
\end{equation*}

\subsection{Globalization}
\label{subsec:glob}

In view of Lemmas~\ref{lem:unif} and \ref{lem:CVL2}, it suffices to
 prove a priori estimates on $\|\Phi\|_{X_T^{2,2}}$, showing that this
 quantity is locally bounded in $T$, to infer that $ \Phi\in
 X_{T}^{2,2}$ for all $T>0$, and then globalize the
 solution by the standard ODE alternative. 
 \smallbreak

 We use the
conservation of the total energy, whose expression we develop:
\begin{align*}
  E(t)&=\(H_x\phi^x(t),\phi^x(t)\)_{L^2_x}+
        \(H_y\phi^y(t),\phi^y(t)\)_{L^2_y}\\
  &\quad+
  \iint_{\R^{d_1}\times\R^{d_2}}w(x,y)|\phi^x(t,x)|^2|\phi^y(t,y)|^2dxdy\\
  &= \frac{1}{2}\|\nabla_x \phi^x(t)\|_{L^2(\R^{d_1})}^2
  +\int_{\R^{d_1}}V_1(x) |\phi^x(t,x)|^2dx
  + \frac{1}{2}\|\nabla_y
    \phi^y(t)\|_{L^2(\R^{d_2})}^2 \\
  &\quad
  +\int_{\R^{d_2}}V_2(y) |\phi^y(t,y)|^2dy +
  \iint_{\R^{d_1}\times\R^{d_2}}w(x,y)|\phi^x(t,x)|^2|\phi^y(t,y)|^2dxdy.
\end{align*}
Since $c_0<1$ in {\bf (H2)}, we infer
\begin{align*}
  E(t)&\ge  \frac{1}{2}\|\nabla_x \phi^x(t)\|_{L^2(\R^{d_1})}^2 +\frac{1}{2}\|\nabla_y
    \phi^y(t)\|_{L^2(\R^{d_2})}^2 
 + (1-c_0)\int_{\R^{d_1}}V_1(x)
        |\phi^x(t,x)|^2dx\\
  &\quad+(1-c_0)\int_{\R^{d_2}}V_2(y) |\phi^y(t,y)|^2dy
  -c_0C\int_{\R^{d_1}} |\phi^x(t,x)|^2dx -c_0C\int_{\R^{d_2}} |\phi^y(t,y)|^2dy.
\end{align*}
The conservations established above yield
\begin{align*}
  \frac{1}{2}\|\nabla_x \phi^x(t)\|_{L^2(\R^{d_1})}^2 &+\frac{1}{2}\|\nabla_y
    \phi^y(t)\|_{L^2(\R^{d_2})}^2 
 + (1-c_0)\int_{\R^{d_1}}V_1(x)
        |\phi^x(t,x)|^2dx\\
  &+(1-c_0)\int_{\R^{d_2}}V_2(y) |\phi^y(t,y)|^2dy\le E(0)+2c_0C.
\end{align*}
This is the coercivity property announced in the introduction, showing
that there exists $M$ depending only on $\|\Phi_0\|_{1,1}$  such that
\begin{equation*}
  \|\Phi\|_{X_T^{1,1}}\le M,
\end{equation*}
for any interval $[0,T]$ on which the solution is well-defined. 
Proceeding like in the proof of Lemma~\ref{lem:unif}, we have
 \begin{equation}\label{eq:reed}
  \sup_{t\in[0,T]} \| H_x \phi^x(t)\| _{L^2_x}^2 \le \| H_x
  \phi_{0}^x\| _{L^2_x}^2 + 2\int_0^T \left|\( [ H_x, \langle
  w\rangle_y(t)] \phi^x(t), H_x \phi^x(t)\) _{L^2_x}\right|dt.
\end{equation}
In view of Lemma~\ref{lem:recursive} with $k=1$, $f=\phi^x$ and
$g=H_x\phi^x$, we infer
\begin{align*}
  \sup_{t\in[0,T]} \| H_x \phi^x(t)\| _{L^2_x}^2 & \le \| H_x
  \phi_{0}^x\| _{L^2_x}^2+ C \|\phi^y\|^2_{L^\infty_T \mathcal
   H^1_y}\int_0^T \|\phi^x(t)\|_{\mathcal H^2_x}
  \| H_x\phi^x(t)\| _{L^2_x}dt\\
  & \le \| H_x
  \phi_{0}^x\| _{L^2_x}^2+ C M^2\int_0^T \|\phi^x(t)\|_{\mathcal H^2_x}^2
  dt.
\end{align*}
The conservation of the $L^2$-norm of $\phi^x$ implies
\begin{equation*}
   \sup_{t\in[0,T]} \|\phi^x(t)\| _{\mathcal H^2_x}^2 \le \| 
  \phi^x\| _{\mathcal H^2_x}^2+ C M^2\int_0^T \|\phi^x(t)\|_{\mathcal H^2_x}^2
  dt,
\end{equation*}
hence an exponential a priori control of the $\mathcal H^2_x$-norm of
$\phi^x(t)$ by Gronwall lemma. The same holds for $\phi^y(t)$, hence
the conclusion of Theorem~\ref{theo:main}.


\section{Proof of
  Lemma~\ref{ex:extra}}\label{sec:sufficient}

We briefly explain why \eqref{hyp:stronger} implies
{\bf (H3)}$_{2,2}$, thanks to an integration by parts, in view of
\eqref{eq:temperance}. Typically, for $f_1,f_2\in \Sch(\R^{d_1})$, 
\[
  \< [w(x,y) , H_x] f_1,f_2\>_{L^2_x} =  \frac{1}{2}\<\Delta_x
  w (\cdot,y)f_1 ,f_2\>_{L^2_x}  + \< \nabla_x w(\cdot,y) \cdot \nabla f_1,   f_2 \>_{L^2_x}  .
  \]
  Therefore, for almost all $y\in\R^{d_2}$, Cauchy-Schwarz inequality yields
  \begin{align*}
 \left|   \< [w(x,y) , H_x] f_1,f_2 \>_{L^2_x}  \right| 
& \le \frac 12  \| |\Delta_xw(\cdot,y)|^{1/2}f_1\|_{L^2_x} 
                                                           \| |\Delta_xw(\cdot,y)|^{1/2}f_2\|_{L^2_x} \\
    &\quad+ \| \nabla_x f_1\|_{L^2_x} \|  \nabla_x w(\cdot,y)
      f_2\|_{L^2_x}. 
\end{align*}
Using 
 \eqref{hyp:stronger},
  \begin{align*}
  \| |\Delta_xw(\cdot,y)|^{1/2}f\|^2_{L^2_x} &   \lesssim \| \sqrt{ V_1} f\|^2_{L^2_x} + (1+V_2(y)) \| f\|^2_{L^2_x}
   \lesssim \|f\|^2_{\mathcal H^1_x}  + (1+V_2(y)) \| f\|^2_{L^2_x},\\
 \|  \nabla_x w(\cdot,y)  f\|_{L^2_x}& \lesssim \|
 (\sqrt{V_1}+V_2(y)+1)f\|_{L^2_x}
   \lesssim  \|f\|_{\mathcal H^1_x}  + (1+V_2(y)) \| f\|_{L^2_x}.
  \end{align*}
  We deduce the expected  relation for $k=\ell=1$:
  \[
 \left|   \< [w(x,y) , H_x] f_1,f_2 \>_{L^2_x} \right| 
 \lesssim (1+V_2(y)) \|f_1\|_{\mathcal H^1_x} \| f_2\|_{\mathcal H^1_x}.
 \]
 For $k=2$, write
 \begin{align*}
 \left|   \< H_x  [w(x,y) , H_x] f_1,f_2 \>_{L^2_x} \right| &  =  \left|   \<
 [w(x,y) , H_x] f_1, H_x f_2 \>_{L^2_x}\right| \\
   &\le \frac{1}{2}\left| \<\Delta_x
  w (\cdot,y)f_1 ,H_xf_2\>_{L^2_x} \right| + \left|\< \nabla_x w(\cdot,y) \cdot \nabla f_1,  H_x f_2 \>_{L^2_x} \right|\\
 & \lesssim \|  \(1+V_1+ V_2(y)\)f_1\|_{L^2_x} \| H_x f_2\|_{L^2_x}\\
  &\quad +
   \|\nabla_x w(\cdot,y) \cdot \nabla f_1\|_{L^2_x}\| H_x f_2\|_{L^2_x}\\
   & \lesssim \|  f_1\|_{\mathcal H^2_x} \|
    f_2\|_{\mathcal H^2_x} +  V_2(y)\|  f_1\|_{L^2_x} \|
     f_2\|_{\mathcal H^2_x}\\
   &\quad+ \|\nabla_x w(\cdot,y) \cdot \nabla
     f_1\|_{L^2_x}\|f_2\|_{\mathcal H^2_x},
 \end{align*}
 where we have used the estimate $\| H_x f\|_{L^2_x} \le\|
 f\|_{\mathcal H_x^2}$. For the last term, \eqref{hyp:stronger} yields
 \begin{align*}
   \|\nabla_x w(\cdot,y) \cdot \nabla
     f_1\|_{L^2_x}&\lesssim \left\|\(\sqrt{V_1}+V_2(y)+1\) \nabla
                    f_1\right\|_{L^2_x}\\
   &\lesssim \left\|\sqrt{V_1} \nabla
     f_1\right\|_{L^2_x}+\(V_2(y)+1\)\|\nabla f_1\|_{L^2_x} \\
   &\lesssim \left\|\sqrt{V_1} \nabla
     f_1\right\|_{L^2_x}+\(V_2(y)+1\)\|f\|_{L^2_x}^{1/2}\|\Delta
     f\|_{L^2_x}^{1/2}\\
   &\lesssim \left\|\sqrt{V_1} \nabla
     f_1\right\|_{L^2_x}+\(V_2(y)+1\)\|f_1\|_{\mathcal H^2_x}.
 \end{align*}
For the first term on the last right hand side, we use
 an integration by parts:
\begin{align*}
  \left\|\sqrt{V_1}\nabla f_1\right\|_{L^2_x}^2 &= \int_{\R^{d_1}}V_1(x) \nabla
                                     f_1(x)\cdot \nabla f_1(x)dx \\
  & = -\int_{\R^{d_1}}V_1(x) 
                                     f_1(x)\Delta f_1(x)dx -
    \int_{\R^{d_1}} f_1(x)\nabla V_1(x) \cdot \nabla f_1(x)dx .
\end{align*}
By Cauchy-Schwarz inequality, the first term on the right hand side is
estimated by
\begin{equation*}
  \|V_1f_1\|_{L^2_x}\|\Delta f_1\|_{L^2_x} \le 2 \|H_x f_1\|_{L^2_x}^2.
\end{equation*}
Invoking \eqref{eq:temperance},
and using Cauchy-Schwarz inequality again,
\begin{align*}
 \left| \int_{\R^{d_1}} f(x)\nabla V_1(x) \cdot \nabla
  f(x)dx\right|&\lesssim \int_{\R^{d_1}} (1+V_1(x))|f(x)| | \nabla f(x)|dx
                 \lesssim \| (1+V_1)f\|_{L^2_x}\|\nabla f\|_{L^2_x}\\
  &\lesssim \(\|f\|_{L^2_x}+\|H_x f\|_{L^2_x}\) \|f\|_{L^2_x}^{1/2}\|\Delta
    f\|_{L^2_x}^{1/2}\\
  & \lesssim  \|H_x f\|_{L^2_x}^{1/2}
    \|f\|_{L^2_x}^{3/2}+\|H_x f\|_{L^2_x}^{3/2}
    \|f\|_{L^2_x}^{1/2}\lesssim \|f\|_{L^2_x}^2 + \|H_x f\|_{L^2_x}^2,
\end{align*}
where we have used Young inequality for the last estimate.
 \begin{align*}  
 \left|   \< [w(x,y) , H_x] H_x f_1, f_2 \> \right| & = \left|   \< H_x f_1, [w(x,y) , H_x] f_2 \> \right| 
  \le\| H_x f_1\|_{L^2_x}  \| \Delta_x w(\cdot,y) f_2\|_{L^2_x} .
 \end{align*}
 To estimate $\<[w(\cdot,y),H_x]H_x f_1,f_2\>_{L^2_x}$, we use the
 self-adjointness of $H_x$ and write
 \begin{equation*}
   \<[w(\cdot,y),H_x]H_x f_1,f_2\>_{L^2_x}=\<H_x f_1,[w(\cdot,y),H_x]f_2\>_{L^2_x}.
 \end{equation*}
 We use the above estimate, where the roles of $f_1$ and $f_2$ have
 been swapped, to conclude that the first inequality in {\bf
   (H3)}$_{2,2}$ holds. The proof of the second one is
 similar. 
 

\appendix

\section{Tangent space}\label{sec:tangent}
For completeness, we give the elementary considerations for determining the tangent spaces 
of the Hartree manifold, Lemma~\ref{lem:tangent}.

\begin{proof}
We consider a curve $\Gamma(s) = \varphi^x(s)\otimes \varphi^y(s)\in\mathcal M$ with $\Gamma(0)=u$.
Then, 
\[
\dot\Gamma(0) =\dot \varphi^x(0)\otimes\varphi^y 
+ \varphi^x\otimes\dot \varphi^y(0),
\]
which verifies the claimed representation of any tangent function as 
\[
v=v^x\otimes\varphi^y + \varphi^x\otimes v^y.
\] 
Let us  consider $a=(a^x,a^y)\in\C^2$ 
with $a^x+a^y = 0$. We set $w^x = v^x + a^x \varphi^x$ and $w^y = v^y + a^y \varphi^y$. 
Then, $w = w^x\otimes\varphi^y + \varphi^x\otimes w^y$ satisfies
\[
w = v^x\otimes\varphi^y + \varphi^x\otimes v^y + (a^x+a^y)\varphi^x\otimes\varphi^y = v.
\]
Choosing $a^x = -\langle \varphi^x,v^x\rangle /\langle \varphi^x,\varphi^x\rangle$ 
and $a^y = -a^x$, we obtain a representation of $v$ satisfying the claimed gauge condition. 
We verify that this condition implies uniqueness. 
We assume that $v = v^x\otimes\varphi^y + \varphi^x\otimes v^y = 
\tilde v^x\otimes\varphi^y + \varphi^x\otimes \tilde v^y$ with 
$\langle \varphi^x,v^x\rangle = \langle \varphi^x,\tilde v^x\rangle =0$. Then, 
for any $\vartheta^y\in L^2_y$,
\[
\langle\varphi^x\otimes\vartheta^y,v\rangle 
= \langle\varphi^x,\varphi^x\rangle \langle\vartheta^y,v^y\rangle
= \langle\varphi^x,\varphi^x\rangle \langle\vartheta^y,\tilde v^y\rangle,
\]
which implies $v^y = \tilde v^y$. Then, for any $\vartheta^x\in L^2_y$,
\begin{align*}
\langle\vartheta^x\otimes\varphi^y,v\rangle_{L^2_{x,y}} 
&= \langle\vartheta^x,v^x\rangle_{L^2_x} \langle\varphi^y,\varphi^y\rangle_{L^2_y} +  
\langle\vartheta^x,\varphi^x\rangle_{L^2_x} \langle\varphi^y,v^y\rangle_{L^2_y}\\
&= \langle\vartheta^x,\tilde v^x\rangle_{L^2_x} \langle\varphi^y,\varphi^y\rangle_{L^2_y} +  
\langle\vartheta^x,\varphi^x\rangle_{L^2_x} \langle\varphi^y,\tilde v^y\rangle_{L^2_y}, 
\end{align*}
which implies $v^x = \tilde v^x$. Choosing $v^x = 0$ and $v^y = \varphi^y$, we have $v=u$ 
so that $u\in\mathcal T_u\mathcal M$.
\end{proof}

\section{Coulombic type coupling}
\label{sec:coulomb}


We recall standard definition and
results.
\begin{definition}[Admissible pairs in $\R^3$]\label{def:adm}
 A pair $(q,r)$ is admissible if $q,r\ge  2$,
 and
\[\frac{2}{q}= 3\left( \frac{1}{2}-\frac{1}{r}\right).\]
\end{definition}
As the range allowed for $(q,r)$ is compact, we set, for $I\subset \R$
a time interval,
\begin{equation*}
  \|u\|_{S(I)} = \sup_{(q,r)\text{
      admissible}}\|u\|_{L^q(I;L^r(\R^3))}. 
\end{equation*}
In view of \cite{Fujiwara} and \cite{KT}, we have:

\begin{proposition}\label{prop:strichartz}
Let $d=3$ and $\mathbf V\in \mathcal Q$. Denote
$\mathbf H=-\frac{1}{2}\Delta + \mathbf V$.\\
  $(1)$ There exists $C_{\rm hom}$  such that for all interval $I$ such that $|I|\le 1$,
\[
\|e^{-it\mathbf H} \varphi\|_{S(I)} \le C_{\rm hom}
\|\varphi \|_{L^2},\quad \forall \varphi\in L^2(\R^3).
\]
$(2)$ Denote
\begin{equation*}
  D(F)(t,x) = \int_0^t e^{-i(t-\tau)\mathbf H}F(\tau,x)\mathrm{d}\tau.
\end{equation*}
There exists $C_{\rm inhom}$
    such that for all interval $I\ni 0$ such that $|I|\le 1$,
\begin{equation*}
      \left\lVert D(F)
      \right\rVert_{S(I)}\le C_{\rm inhom} \left\lVert
      F\right\rVert_{S(I)^*}.
\end{equation*}
\end{proposition}

The existence of (local in time) Strichartz estimates of
Proposition~\ref{prop:strichartz} is the main ingredient of the proof
of Theorem~\ref{theo:CoulombStrichartz}. Actually, as soon as such
estimates are available for the operators $\mathbf H_x$ and $\mathbf
H_y$, then 
Theorem~\ref{theo:CoulombStrichartz} remains valid.
As mentioned in the introduction, such cases can be found in
e.g. \cite{RodnianskiSchlag} or \cite{BPST04}. On the other hand, we
emphasize that for superquadratic potentials, like   $V_1$ in
Example~\ref{ex:companion},
Strichartz estimates suffer a loss of regularity; see \cite{Miz14,YajZha04}.

\begin{remark}
  The case of a harmonic potential, $\mathbf V(x)=|x|^2$, shows that
  $\mathbf H$ may have eigenvalues, and explains why the above time
  intervals $I$ are required to have finite length. 
\end{remark}
\begin{remark}
  The potential $\mathbf V$ may also be time dependent, in view the
   original framework of \cite{Fujiwara}:
  $\mathbf V\in L^\infty_{\rm loc}(\R_t\times \R_x^3)$ is real-valued, and 
  smooth with respect 
  to the space variable: for (almost) all $t\in \R$, $x\mapsto \mathbf
  V(t,x)$
  is a $C^\infty$ map. Moreover, it is at most quadratic in space: 
  \begin{equation*}
   \forall T>0,\quad  \forall \alpha \in \N^d, \ |\alpha|\ge 2, \quad
    \d_x^\alpha \mathbf V\in L^\infty([-T,T]\times \R_x^3). 
  \end{equation*}
Under these assumptions, suitable modifications of
Proposition~\ref{prop:strichartz} are needed, but they do not alter
the conclusion of Theorem~\ref{theo:CoulombStrichartz} (see
\cite{Ca11}). See also \cite{RodnianskiSchlag} for another class of
time dependent potentials. 
\end{remark}
\begin{proof}[Proof of Theorem~\ref{theo:CoulombStrichartz}]
We give the main technical steps of the proof,  and refer to  \cite{CazCourant}
for details.
 By Duhamel's formula, we write \eqref{eq:system} as
\begin{align*}
\phi^x(t)&=e^{-it\mathbf H_x}\phi_0^x-
i\int_0^{t}e^{-i(t-\tau)\mathbf H_x}\(v_1\phi^x+ \(W \ast |\phi^y|^{2}\)\phi^x\)(\tau)d\tau=:F_1(\phi^x,\phi^y),\\
\phi^y(t)&=e^{-it\mathbf H_y}\phi_0^y-
i\int_0^{t}e^{-i(t-\tau)\mathbf H_y}\(v_2\phi^y+\(W \ast
           |\phi^x|^{2}\)\phi^y\)(\tau)d\tau=:F_2(\phi^x,\phi^y). 
\end{align*}
Theorem~\ref{theo:CoulombStrichartz} follows from a standard fixed
point argument based on Strichartz estimates. For $0<T\le 1$, we 
introduce
 \begin{equation*}
\begin{aligned}
Y(T)&=\{(\phi^x,\phi^y)\in C([0,T]; L^{2}(\R^3))^2:\quad
 \|\phi^x\|_{S([0,T])}\le
2C_{\rm hom}\|\phi_0^x\|_{L^{2}},\\
&\qquad \|\phi^y\|_{S([0,T])}\le
2C_{\rm hom}\|\phi_0^y\|_{L^{2}}\},
\end{aligned}
\end{equation*}
and the distance 
\[d(\phi_{1},
\phi_{2})=\|\phi_{1}-\phi_{2}\|_{S([0,T]},\]
where $C_{\rm hom}$ stems from Proposition~\ref{prop:strichartz}.
Then $(Y(T), d)$ is a complete metric space.

 By using Strichartz estimates and H\"older inequality, we have:
\begin{equation*}
\begin{aligned}
\|F_1(\phi^x,\phi^y)\|_{S([0,T])}&\le C_{\rm hom}
\|\phi_0^x\|_{L^{2}}+C_{\rm inhom}\(\|v_1\phi^x\|_{S([0,T])^*}+\left\|\(W\ast|\phi^y|^2\)\phi^x\right\|_{S([0,T])^*}\) ,
\end{aligned}
\end{equation*}
 for any $(\phi^x,\phi^y)\in Y(T)$. By assumption (see
 Theorem~\ref{theo:CoulombStrichartz}), we may write 
 \begin{equation*}
   v_1=v_1^p+v_1^\infty,\quad
   v_2=v_2^p+v_2^\infty,\quad W=W^p+W^\infty,\quad v_1^q,v_2^q,W^q\in
   L^q(\R^3), 
 \end{equation*}
and the value $p$ can obviously be the same for the three potentials,
by taking the minimum between the three $p$'s if needed. Regarding
$\|v_1\phi^x\|_{S([0,T])^*}$, we write 
\begin{align*}
  \|v_1^\infty \phi^x\|_{S([0,T])^*}&\le \|v_1^\infty
  \phi^x\|_{L^1([0,T];L^2)}\le \|v_1^\infty\|_{L^\infty} \|
  \phi^x\|_{L^1([0,T];L^2)}\\
&\le T \|v_1^\infty\|_{L^\infty} \|
  \phi^x\|_{L^\infty\le ([0,T];L^2)}\le T \|v_1^\infty\|_{L^\infty} \|
  \phi^x\|_{S([0,T])}.
\end{align*}
Let $r$ be such that
\begin{equation*}
  \frac{1}{r'}=\frac{1}{r}+\frac{1}{p}\Longleftrightarrow 1= \frac{2}{r}+\frac{1}{p}.
\end{equation*}
Note that this exponent is the one introduced in the statement of
Theorem~\ref{theo:CoulombStrichartz}. 
The assumption $p>3/2$ implies $2\le r<6$. Let $q$ be such that
$(q,r)$ is 
admissible: $r<6$ implies $q>2$. H\"older inequality yields
\begin{align*}
  \|v_1^p \phi^x\|_{S([0,T])^*}&\le \|v_1^p
  \phi^x\|_{L^{q'}([0,T];L^{r'})}\le \|v_1^p\|_{L^p}
                                 \|\phi^x\|_{L^{q'}([0,T];L^{r})}\\ 
 & \le T^{1/\theta} \|v_1^p\|_{L^p} \|\phi^x\|_{L^{q}([0,T];L^{r})}
  \le T^{1/\theta} \|v_1^p\|_{L^p} \|\phi^x\|_{S([0,T])},
\end{align*}
where $\theta$ is such that
\begin{equation*}
  \frac{1}{q'}=\frac{1}{q}+\frac{1}{\theta}. 
\end{equation*}
Note that $\theta$ is finite, as $q>2$ 
\smallbreak

For the convolution term, first write 
\begin{align*}
    \left\|\(W^\infty \ast
           |\phi^y|^{2}\)\phi^x\right\|_{S([0,T])^*}&\le \left\|\(W^\infty \ast
           |\phi^y|^{2}\) 
  \phi^x\right\|_{L^1([0,T];L^2)}\\
&\le \left\|W^\infty \ast
           |\phi^y|^{2} \right\|_{L^1([0,T];L^\infty)}
 \| \phi^x\|_{L^\infty([0,T];L^2)}\\
& \le T \|W^\infty \|_{L^\infty} \|\phi^y\|_{L^\infty([0,T];L^2)}^2 \|
                                       \phi^x\|_{L^\infty([0,T];L^2)}\\ 
&\le T \|W^\infty \|_{L^\infty} \|\phi^y\|_{S([0,T])}^2 \|
                                       \phi^x\|_{S([0,T])}.
\end{align*}
Introduce $r_1$ such that
\begin{equation}\label{eq:rp}
  \frac{1}{r_1'}=\frac{1}{r_1}+\frac{1}{2p}\Longleftrightarrow
  2=\frac{4}{r_1}+\frac{1}{p}\Longleftrightarrow  1+\frac{1}{2p} =
  \frac{1}{p}+\frac{2}{r_1}. 
\end{equation}
The assumption $p>3/2$ implies $2\le r_1<3$. Let $q_1$ be such that
$(q_1,r_1)$ is admissible: $q_1>4$. H\"older inequality yields
\begin{align*}
    \left\|\(W^p \ast
           |\phi^y|^{2}\)\phi^x\right\|_{S([0,T])^*}&\le \left\|\(W^p \ast
           |\phi^y|^{2}\) 
  \phi^x\right\|_{L^{q_1'}([0,T];L^{r_1'})}\\
&\le  \left\|W^p \ast  |\phi^y|^{2}
  \right\|_{L^{k}([0,T];L^{2p})} \|\phi^x\|_{L^{q_1}([0,T];L^{r_1})}\\
&\le  \left\|W^p \ast  |\phi^y|^{2}
  \right\|_{L^{k}([0,T];L^{2p})} \|\phi^x\|_{S([0,T])},
\end{align*}
where $k$ is such that 
\begin{equation*}
  \frac{1}{q_1'}=\frac{1}{q_1}+\frac{1}{k} \Longleftrightarrow
  1=\frac{2}{q_1}+\frac{1}{k}. 
\end{equation*}
Note that since $q_1>4$, we have $q_1>2k$. In view of \eqref{eq:rp}, Young
inequality yields
\begin{align*}
  \left\|W^p \ast  |\phi^y|^{2}
  \right\|_{L^{k}([0,T];L^{2p})}& \le \|W^p\|_{L^p} \left\| |\phi^y|^{2}
  \right\|_{L^{k}([0,T];L^{r_1/2})} = \|W^p\|_{L^p} \| \phi^y
  \|_{L^{2k}([0,T];L^{r_1})}^2 \\
&\le T^\eta \|W^p\|_{L^p} \| \phi^y
  \|_{L^{q_1}([0,T];L^{r_1})}^2 \le  T^\eta \|W^p\|_{L^p} \| \phi^y
  \|_{S([0,T])}^2, 
\end{align*}
where $\eta>0$ is given by $\eta = 1/(2k)-1/q_1$.
\smallbreak

The same inequalities obviously holds
 by switching $x$ and $y$, and so for $T>0$ sufficiently small,
 $\Phi:=(\phi^x,\phi^y)\mapsto
 (F_1(\phi^x,\phi^y),F_2(\phi^x,\phi^y))=:\mathbf F(\Phi)$ leaves $Y(T)$
 invariant. 
\smallbreak

Using similar estimates, again relying on Strichartz and H\"older
inequalities involving the same Lebesgue exponents ($\mathbf F$ is the sum
of a linear and a trilinear term in $\Phi$), we infer that up
to decreasing $T>0$, $\mathbf F$ is a contraction on $Y(T)$, and so
there exists  a unique $\Phi\in 
Y(T)$ solving \eqref{eq:system}. The global existence of the
solution for \eqref{eq:system} follows from the conservation
of the $L^{2}$-norms of $\phi^x$ and $\phi^y$, respectively. 
\smallbreak

Uniqueness of such solutions follows once again from Strichartz and H\"older
inequalities involving the same Lebesgue exponents as above, like for
the contraction part of the argument. The main remark consists in
noticing that the above Lebesgue indices satisfy $r>r_1$, hence
$q<q_1$, and so
$L_{\rm loc}^{q_1}L^{r_1}\subset L_{\rm loc}^q L^r\cap L^\infty L^2$. 
\end{proof}

\begin{remark}[$H^1$-regularity]\label{rem:coulombH1}
  If in Theorem~\ref{theo:CoulombStrichartz}, we assume in addition
  that 
  \begin{equation*}
    \nabla v_1,\nabla v_2\in L^p(\R^3)+L^\infty(\R^3) \quad \text{form
      some }p>3/2,
  \end{equation*}
then for $\phi_0^x,\phi_0^y\in H^1(\R^3)$ and $x\phi_0^x,y\phi_0^y\in
L^2(\R^3)$ (this last assumption may be removed when $\nabla V_1,\nabla
V_2\in L^\infty(\R^3)$ -- the minimal assumption to work at the
$H^1$-level with $V_1,V_2\in \mathcal Q$  is $\phi_0^x \nabla V_1,\phi_0^y\nabla V_2\in
L^2(\R^3)$, see \cite{Ca15}), the global solution
constructed in Theorem~\ref{theo:CoulombStrichartz} satisfies 
\begin{equation*}
  (\phi^x,\phi^y)\in C(\R;H^1(\R^3))^2\cap
  L^q_{\rm loc}(\R;W^{1,r}(\R^3))^2,\quad (x\phi^x,y\phi^y)\in C(\R;L^2(\R^3)^3)^2.
\end{equation*}
To see this, it suffices to resume the above proof, and check that
$\nabla_x F_1(\phi^x,\phi^y)$ and $ \nabla_y F_2(\phi^x,\phi^y)$
satisfy essentially the same estimates as $F_1(\phi^x,\phi^y), F_2(\phi^x,\phi^y)$
in $S([0,T])$. One first has to commute the gradient with
$e^{-it\mathbf H_x}$ or $e^{-it\mathbf H_y}$. Typically,
\begin{align*}
  \nabla_x F_1 (\phi^x,\phi^y) &= e^{-it\mathbf H_x}\nabla_x\phi_0^x-
i\int_0^{t}e^{-i(t-\tau)\mathbf H_x}\nabla_x \(v_1\phi^x+ \(W \ast
                               |\phi^y|^{2}\)\phi^x\)(\tau)d\tau\\
&\quad -
i\int_0^{t}e^{-i(t-\tau)\mathbf H_x}F_1 (\phi^x,\phi^y) (\tau)\nabla_x V_1d\tau,
\end{align*}
where the last factor accounts for the possible lack of commutation
between $\mathbf H_x$ and $\nabla_x$, $[-i\d_t-\mathbf
H_x,\nabla_x]=\nabla_x V_1$. Since $V_1$ is at most
quadratic, $\nabla V_1$ is at most linear, and we obtain a closed
system of estimates by considering
\begin{align*}
  x F_1 (\phi^x,\phi^y) &= e^{-it\mathbf H_x}(x\phi_0^x)-
i\int_0^{t}e^{-i(t-\tau)\mathbf H_x} \(x \(v_1\phi^x+ \(W \ast
                               |\phi^y|^{2}\)\phi^x\)\)(\tau)d\tau\\
&\quad +
i\int_0^{t}e^{-i(t-\tau)\mathbf H_x}\nabla_xF_1 (\phi^x,\phi^y) (\tau)d\tau,
\end{align*}
where we have used $[-i\d_t-\mathbf H_x,x]= -\nabla_x$. We omit the
details, and refer to \cite{CazCourant} (see also \cite{Ca11}). As pointed in
Remark~\ref{rem:energy}, the energy 
\begin{align*}
  E(t)&=\(H_x\phi^x(t),\phi^x(t)\)_{L^2_x}+
        \(H_y\phi^y(t),\phi^y(t)\)_{L^2_y}\\
  &\quad
+  \iint_{\R^{3}\times\R^{3}}W(x-y)|\phi^x(t,x)|^2|\phi^y(t,y)|^2dxdy,
 \end{align*}
which is well defined with the above regularity, is independent of
time. Formally, this can be seen by multiplying the first equation in
\eqref{eq:system} by $\d_t \phi^x$, the second by $\d_t \phi^y$,
integrating in space, considering the real part, and summing the two
identities. To make the argument rigorous (we may not have enough
regularity to be allowed to proceed as described), one may use a
regularization procedure as in \cite{CazCourant}, or rely on a clever
use of the regularity provided by Strichartz estimates, as in
\cite{Ozawa2006}. 
\end{remark}


\bibliographystyle{abbrv}
\bibliography{hartree}
\end{document}